\documentclass[12pt]{article}
\usepackage{amsthm,amsfonts, amsbsy,amssymb,amsmath,graphicx}
\usepackage{graphics}
\usepackage[english]{babel}
\usepackage{graphicx}
\usepackage{lineno}
\usepackage{enumerate}
\usepackage{tikz}
\usepackage{tkz-graph}
\usepackage{tkz-berge}
\usepackage{hyperref}
\usepackage{color}
\usepackage{tikz}

\hypersetup{colorlinks=true}

\hypersetup{colorlinks=true, linkcolor=blue, citecolor=blue,urlcolor=blue}

\input epsf

\usepackage{epic,latexsym,amssymb}
\usepackage{tikz}

\usepackage{tikz,epsf}

\usepackage{amsfonts}
\usepackage{amscd}
\usepackage{amsmath}
\usepackage{graphicx}
\usepackage{color}
\usepackage{caption,subcaption}

\newtheorem{theorem}{Theorem}[section]
\newtheorem{observation}[theorem]{Observation}
\newtheorem{proposition}[theorem]{Proposition}
\newtheorem{lemma}[theorem]{Lemma}

\newtheorem{corollary}[theorem]{Corollary}

\newcommand{\diam}{{\rm diam}}

\newcommand{\opack}{\rho^{\rm o}}

\def\ST{\widehat{S}}

\tikzstyle{vertex}=[circle, draw, inner sep=0pt, minimum size=6pt]

\begin{document}

\title{Injective colorings of Sierpi\'nski-like graphs and Kneser graphs}

\author{Bo\v{s}tjan Bre\v{s}ar$^{a,b}$\thanks{\texttt{bostjan.bresar@um.si}}
\and Sandi Klav\v zar$^{a,b,c,}$\thanks{\texttt{sandi.klavzar@fmf.uni-lj.si}}
    \and Babak Samadi$^{b,}$\thanks{\texttt{babak.samadi@imfm.si}}
\and Ismael G. Yero$^{d,}$\thanks{\texttt{ismael.gonzalez@uca.es}}
}
\maketitle

\begin{center}
$^a$ Faculty of Natural Sciences and Mathematics, University of Maribor, Slovenia\\
\medskip
	
$^b$ Institute of Mathematics, Physics and Mechanics, Ljubljana, Slovenia\\
\medskip
	
$^c$ Faculty of Mathematics and Physics, University of Ljubljana, Slovenia\\
\medskip

$^d$ Departamento de Matem\'{a}ticas, Universidad de C\'adiz, Algeciras Campus, Spain \\
\medskip
\end{center}

\begin{abstract}
Two relationships between the injective chromatic number and, respectively, chromatic number and chromatic index, are proved. They are applied to determine the injective chromatic number of Sierpi\'nski graphs and to give a short proof that Sierpi\'nski graphs are Class $1$. Sierpi\'nski-like graphs are also considered, including generalized Sierpi\'nski graphs over cycles and rooted products. It is proved that the injective chromatic number of a rooted product of two graphs lies in a set of six possible values. Sierpi\'nski graphs and Kneser graphs $K(n,r)$ are considered with respect of being perfect injectively colorable, where a graph is  perfect injectively colorable if it has an injective coloring in which every color class forms an open packing of largest cardinality. In particular, all Sierpi\'nski graphs and Kneser graphs $K(n, r)$ with $n \ge 3r-1$ are perfect injectively colorable graph, while $K(7,3)$ is not.
\end{abstract}


\noindent
{\bf AMS Subj.\ Class.\ (2020)}:  05C15, 05C69, 05C76

\section{Introduction}

Throughout the paper, we consider $G$ as a finite simple graph with vertex set $V(G)$ and edge set $E(G)$. The (\textit{open}) {\em neighborhood} of a vertex $v$ is denoted by $N_{G}(v)$, and $N_{G}[v]=N_{G}(v)\cup \{v\}$ is its {\em closed neighborhood} (we omit the index $G$ if the graph $G$ is clear from the context). The {\em minimum} and {\em maximum degrees} of $G$ are denoted by $\delta(G)$ and $\Delta(G)$, respectively. For terminology and notation not explicitly defined here, we refer to \cite{we}.

Recall that a (\textit{vertex}) \textit{coloring} of $G$ is a labeling of the vertices of $G$ so that any two adjacent vertices have distinct labels. The \emph{chromatic number} of $G$, denoted $\chi(G)$, is the smallest number of labels in a coloring of $G$. For some additional information on coloring problems, we refer the reader to~\cite{JT}.

A function $f:V(G)\rightarrow\{1,\dots,k\}$ is an {\em injective $k$-coloring} if no vertex $v$ is adjacent to two vertices $u$ and $w$ with $f(u)=f(w)$. For an injective $k$-coloring $f$, the set of color classes $\big{\{}\{v\in V(G)\mid f(v)=i\}:\,1\leq i\leq k\big{\}}$ is also called an \emph{injective $k$-coloring} of $G$ (or simply an {\em injective coloring} if $k$ is clear from the context). The minimum $k$ for which a graph $G$ admits an injective $k$-coloring is the {\em injective chromatic number} of $G$, and is denoted by $\chi_{i}(G)$. An injective $k$-coloring for which $k=\chi_{i}(G)$ is called a $\chi_{i}(G)$-\textit{coloring}. The study of injective coloring was initiated in \cite{hkss}, and then intensively pursued, see, for example,~\cite{bu-2009, cky-2010, lst-2009, pp}. In particular, the injective colorings of some products and graphs operations have been studied in \cite{bkr2,SSY,sy}.

A set $B\subseteq V(G)$ is an {\em open packing} in $G$ if $N(u)\cap N(v)=\emptyset$ for all distinct vertices $u,v\in B$, and the maximum cardinality of an open packing in $G$ is the {\em open packing number}, $\opack(G)$, of $G$. An open packing of cardinality $\opack(G)$ is an {\em $\opack(G)$-set}. The concept was introduced in~\cite{hs-1999}, and was studied in several papers mainly due to its relation with total domination. It was noticed in~\cite{BSY} that an injective coloring of a graph $G$ is equivalent to a partition of $V(G)$ into open packings, i.e., the vertices colored with a same color in the injective coloring form an open packing in $G$. In connection with this, the following concept was introduced in~\cite{BSY}, and further on, partially investigated for hypercubes in~\cite{bkr2}. A graph $G$ is \emph{perfect injectively colorable} if it admits an injective coloring in which every color class forms an open packing of maximum cardinality. Note that such an injective coloring is necessarily a $\chi_i(G)$-coloring.

Section 2 is devoted to two auxiliary lemmas based on establishing some helpful relationships between injective coloring and, respectively, vertex coloring and edge coloring. They will be efficiently used in Section 3 in order to prove for each Sierpi\'nski graph $S_p^n$ that $(i)$ $\chi_i(S_p^n)=p=\Delta(S_p^n)$ with $p\geq3$ and $n\geq1$, and that $(ii)$ $S_p^n$ belongs to Class $1$ with $n,p\geq2$. (Recall that a simple graph $G$ is Class 1 if $\chi'(G)=\Delta(G)$, in which $\chi'$ stands for the edge-chromatic number.) Note that the assertion $(ii)$ was already proved by Hinz and Parisse in 2012 (\cite{HP-2012}). However, our proof is much shorter based on the present, different approach.

Injective coloring of the rooted product graph $G\circ_{v}H$, as a Sierpi\'nski-type product graph, is discussed in Section 4. It is readily seen that $\chi_{i}(G\circ_{v}H)$ can be bounded from below and above by $\max\{\chi_i(G),\chi_i(H)\}$ and $\chi_i(G)+\chi_i(H)$, respectively. We prove that $\chi_{i}(G\circ_{v}H)$ only assumes $6$ values from this interval. And, as an immediate result, this leads to a closed formula for this parameter in the case of corona product graphs given in \cite{SSY}.

We also investigate the perfect injectively colorability of Sierpi\'nski graphs and Kneser graphs. It is proved that each Sierpi\'nski graph $S_p^n$ with $p\geq3$ and $n\geq1$ is perfect injectively colorable, while this is not the case for generalized Sierpi\'nski graphs by giving a special counterexample. Finally, we prove that all Kneser graphs $K(n,r)$ with $n\geq3r-1$ are perfect injectively colorable. Moreover, this is a best possible result as the Kneser graph $K(7,3)$ does not satisfy this property.

\section{Two lemmas on injective colorings versus (edge) colorings}
\label{sec:two-coloring-connections}

In this section we prove two relationships between the injective chromatic number and, respectively, chromatic number and chromatic index. Their proofs are not difficult, but we will later demonstrate that the results can be very useful.

For the first result,  consider the following concept. Let $G$ be a graph. A collection ${\cal C}= \{C_1,\ldots,C_k\}$ of cliques in $G$ is an {\em edge clique cover} of $G$ if every edge of $G$ belongs to some $C_i\in {\cal C}$. For more information on edge clique covers see the survey~\cite{schwartz-2022} and recent papers~\cite{chu-2024, nguyen-2024}. We say that an edge clique cover $\cal C$ is {\em sparse} if every vertex of $G$ belongs to at most two cliques in $\cal C$. Note that not every graph has a sparse edge clique cover.  For instance, among the triangle-free graphs $G$ only the graphs with $\Delta(G)\le 2$ admit sparse edge clique covers.

Let $\cal G$ be the class of graphs that admit a sparse edge clique cover.  If $G\in \cal G$ and ${\cal C}= \{C_1,\ldots,C_k\}$ is a sparse edge clique cover of $G$, then we introduce the graph $G^{\cal C}$ constructed from $G$ as follows. First, considering the vertex sets of the cliques $C_i$ to be pairwise disjoint in $G^{\cal C}$, we set $V(G^{\cal C})=\bigcup_{i=1}^k{V(C_i)}$. Note that by this convention, $|V(G^{\cal C})| = \sum_{i=1}^k |V(C_i)|$. Second, two vertices in $G^{\cal C}$ are adjacent if they are either in the same clique from $\cal C$ or they correspond to the same vertex from two cliques of ${\cal C}$. See Fig.~\ref{fig:example-GC} for an example of this construction.

\begin{figure}[ht!]
\begin{center}
\begin{tikzpicture}[scale=1.0,style=thick,x=1cm,y=1cm]
\def\vr{3pt}
\begin{scope}[xshift=0cm, yshift=0cm] 
\coordinate(x1) at (-0.7,0.5);
\coordinate(x2) at (0,0);
\coordinate(x3) at (0,1);
\coordinate(x4) at (1,0);
\coordinate(x5) at (1,1);
\coordinate(x6) at (2,0);
\coordinate(x7) at (2,1);
\coordinate(x8) at (2.7,0.5);
\draw (x1) -- (x2) -- (x4) -- (x6) -- (x7) -- (x5) -- (x3) -- (x1);
\draw (x2) -- (x3) -- (x4) -- (x5) -- (x2);
\draw (x4) -- (x7) -- (x8);
\draw (x5) -- (x6);
\foreach \i in {1,...,8}
{
\draw(x\i)[fill=white] circle(\vr);
}
\node at (1,-1) {$G$};
\draw (-0.3,0.5) ellipse (0.7cm and 1.0cm);
\draw (0.5,0.5) ellipse (0.85cm and 1.0cm);
\draw (1.5,0.5) ellipse (0.85cm and 1.0cm);
\draw[rotate=-33] (1.6,1.9) ellipse (0.85cm and 0.4cm);
\end{scope}
\begin{scope}[xshift=5cm, yshift=0cm] 
\coordinate(x1) at (-0.7,0.5);
\coordinate(x2) at (0,0);
\coordinate(x3) at (0,1);
\coordinate(x4) at (1,0);
\coordinate(x5) at (1,1);
\coordinate(x6) at (2,0);
\coordinate(x7) at (2,1);
\coordinate(x8) at (3,0);
\coordinate(x9) at (3,1);
\coordinate(x10) at (4,0);
\coordinate(x11) at (4,1);
\coordinate(x12) at (5,1);
\coordinate(x13) at (5.7,0.5);
\draw (x1) -- (x2) -- (x4) -- (x6) -- (x8) -- (x10) -- (x11) -- (x9)-- (x7) -- (x5) -- (x3) -- (x1);
\draw (x4) -- (x5) -- (x6) -- (x7) -- (x4);
\draw (x8) -- (x11) -- (x10) -- (x9) -- (x8);
\draw (x2) -- (x3);
\draw (x11) -- (x12) -- (x13);
\foreach \i in {1,...,13}
{
\draw(x\i)[fill=white] circle(\vr);
}
\node at (2.5,-1) {$G^{\cal C}$};
\draw (-0.3,0.5) ellipse (0.7cm and 1.0cm);
\draw (1.5,0.5) ellipse (0.85cm and 1.0cm);
\draw (3.5,0.5) ellipse (0.85cm and 1.0cm);
\draw[rotate=-33] (4.0,3.5) ellipse (0.85cm and 0.4cm);
\end{scope}
\end{tikzpicture}
\caption{A graph $G$, its sparse edge clique cover ${\cal C}$ (consisting of the circled cliques), and the derived graph $G^{\cal C}$}
\label{fig:example-GC}
\end{center}
\end{figure}
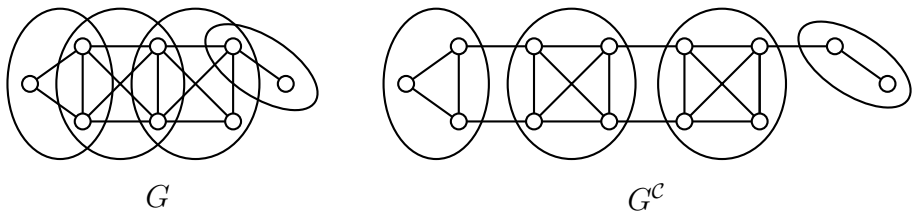

Our first result now reads as follows.

\begin{lemma}
\label{lem:vertexcoloring}
If $G\in {\cal G}$ with a sparse edge clique cover ${\cal C}= \{C_1,\ldots,C_k\}$, then $\chi_i(G^{\cal C})\le \chi(G)$.
\end{lemma}

\begin{proof} Let $c:V(G)\to [k]$ be a proper coloring of $G$, where $k=\chi(G)$. Let the coloring $c^{*}: V(G^{\cal C})\to [k]$ be defined by $c^{*}(x)=c(x)$ if $x\in V(G)$ belongs to only one clique of $\cal C$, and $c^{*}(x_i)=c(x)=c^{*}(x_j)$, when $x$ belongs to $C_i$ and $C_j$. See Fig.~\ref{fig:example-colorings} for an example of such a derived coloring.

\begin{figure}[ht!]
\begin{center}
\begin{tikzpicture}[scale=1.0,style=thick,x=1cm,y=1cm]
\def\vr{3pt}
\begin{scope}[xshift=0cm, yshift=0cm] 
\coordinate(x1) at (-0.7,0.5);
\coordinate(x2) at (0,0);
\coordinate(x3) at (0,1);
\coordinate(x4) at (1,0);
\coordinate(x5) at (1,1);
\coordinate(x6) at (2,0);
\coordinate(x7) at (2,1);
\coordinate(x8) at (2.7,0.5);
\draw (x1) -- (x2) -- (x4) -- (x6) -- (x7) -- (x5) -- (x3) -- (x1);
\draw (x2) -- (x3) -- (x4) -- (x5) -- (x2);
\draw (x4) -- (x7) -- (x8);
\draw (x5) -- (x6);
\foreach \i in {1,...,8}
{
\draw(x\i)[fill=white] circle(\vr);
}
\node at (1,-1) {$G$};
\draw[below] (x1)++(0,-0.1) node {$1$};
\draw[below] (x2)++(0,-0.1) node {$2$};
\draw[above] (x3)++(0,0.1) node {$3$};
\draw[below] (x4)++(0,-0.1) node {$1$};
\draw[above] (x5)++(0,0.1) node {$4$};
\draw[below] (x6)++(0,-0.1) node {$3$};
\draw[above] (x7)++(0,0.1) node {$2$};
\draw[below] (x8)++(0,-0.1) node {$1$};
\end{scope}
\begin{scope}[xshift=5cm, yshift=0cm] 
\coordinate(x1) at (-0.7,0.5);
\coordinate(x2) at (0,0);
\coordinate(x3) at (0,1);
\coordinate(x4) at (1,0);
\coordinate(x5) at (1,1);
\coordinate(x6) at (2,0);
\coordinate(x7) at (2,1);
\coordinate(x8) at (3,0);
\coordinate(x9) at (3,1);
\coordinate(x10) at (4,0);
\coordinate(x11) at (4,1);
\coordinate(x12) at (5,1);
\coordinate(x13) at (5.7,0.5);
\draw (x1) -- (x2) -- (x4) -- (x6) -- (x8) -- (x10) -- (x11) -- (x9)-- (x7) -- (x5) -- (x3) -- (x1);
\draw (x4) -- (x5) -- (x6) -- (x7) -- (x4);
\draw (x8) -- (x11) -- (x10) -- (x9) -- (x8);
\draw (x2) -- (x3);
\draw (x11) -- (x12) -- (x13);
\foreach \i in {1,...,13}
{
\draw(x\i)[fill=white] circle(\vr);
}
\node at (2.5,-1) {$G^{\cal C}$};
\draw[below] (x1)++(0,-0.1) node {$1$};
\draw[below] (x2)++(0,-0.1) node {$2$};
\draw[above] (x3)++(0,0.1) node {$3$};
\draw[below] (x4)++(0,-0.1) node {$2$};
\draw[above] (x5)++(0,0.1) node {$3$};
\draw[below] (x6)++(0,-0.1) node {$1$};
\draw[above] (x7)++(0,0.1) node {$4$};
\draw[below] (x8)++(0,-0.1) node {$1$};
\draw[above] (x9)++(0,0.1) node {$4$};
\draw[below] (x10)++(0,-0.1) node {$3$};
\draw[above] (x11)++(0,0.1) node {$2$};
\draw[above] (x12)++(0,0.1) node {$2$};
\draw[below] (x13)++(0,-0.1) node {$1$};

\end{scope}
\end{tikzpicture}
\caption{An optimal coloring of $G$ and an optimal injective coloring of $G^{\cal C}$}
\label{fig:example-colorings}
\end{center}
\end{figure}
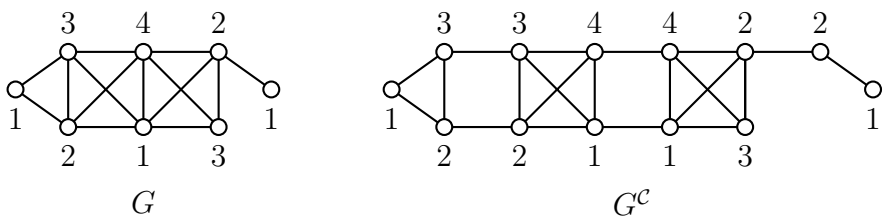

We claim that $c^{*}$ is injective, so we need to show that no two vertices in the neighborhood of any vertex $x$ in $G^{\cal C}$ receive the same color. Assume that $x$ is a vertex of $G$ that lies in only one clique, say $C_i$, of $\cal C$. Then, its neighbors are all in $C_i$, and since $c$ is a proper coloring, they have pairwise different colors (according to $c^{*}$ and $c$). Secondly, assume that $x=x_i$, that is, $x$ belongs to two cliques of $G$ from $\cal C$, one of which being $C_i$, and let the other be $C_j$. In this case, $N_{G^{\cal C}}(x_i)=V(C_i)\cup\{x_j\}$.  Then, the neighbors of $x_i$ in $C_i$ have pairwise different colors (by $c$ and $c^{*}$), which are all different from $c(x)=c^{*}(x_i)=c^{*}(x_j)$. In both cases, all neighbors of $x$ get pairwise different colors by $c^{*}$.
\end{proof}

Our second result detects a new family of Class 1 graphs based on their injective chromatic number. (Recall that by Vizing's theorem, $\Delta(G)\le \chi'(G)\le \Delta(G)+1$ holds for any graph $G$, where the graphs achieving the lower bound are said to belong to Class $1$.)

\begin{lemma}
\label{lem:edgecoloring}
If $\chi_i(G)=\Delta(G)$, then $G$ belongs to Class $1$.
\end{lemma}

\begin{proof}
Let $\Delta=\Delta(G)$, and let $c:V(G)\to \{0,1,\ldots,\Delta-1\}$ be an injective coloring of $G$. Define an edge coloring $c'=E(G)\to \{0,1,\ldots,\Delta-1\}$ arising from $c$ as follows. For each edge $uv\in E(G)$, let $c'(uv)=c(u)+c(v)\pmod\Delta $.  To see that $c'$ is a proper edge coloring of $G$, let $uv$ and $uw$ be two incident edges in $G$. Since $c(v)\ne c(w)$, we infer that $c'(uv)=c(u)+c(v) \ne c(u)+c(w) =c'(uw)$, where summations are taken with respect to modulo $\Delta$. Thus, $\chi'(G)\le \Delta$, which by Vizing's theorem implies that $G$ is in Class 1.
\end{proof}

\section{Sierpi\'nski graphs}
\label{sec:Sierpinski}

This section is devoted to obtain the injective chromatic number of Sierpi\'nski graphs, and to give some consequences of these computations. In particular, we give a short proof of the fact that all Sierpi\'nski graphs belong to Class~1, a result which was first proved in~\cite{HP-2012} with a lengthy argument.

When the Switching Tower of Hanoi game was introduced in~\cite{klavzar-1997}, it was natural to introduce Sierpi\'nski graphs. This family of graphs has subsequently attracted a great deal of interest for various reasons, see a very comprehensive 2017 survey paper~\cite{hinz-2017}, where, in addition to an overview of the results on the Sierpi\'nski graphs,  a classification of Sierpi\'nski-type graphs is proposed. For some later related papers we refer to~\cite{balakrishnan-2022, deng-2021, menon-2023}.

For $p\ge 1$ and $n\ge 1$, the {\em Sierpi\'nski graph} $S_p^n$ has the vertex set $V(S_p^n) = [p]^n$, and two vertices $(u_1,\ldots, u_n)$ and $(v_1,\ldots, v_n)$ are adjacent if there exists an index $d\in [n]$ such that (i) $u_i = v_i$ for $i\in [d-1]$, (ii) $u_d\ne v_d$, and (iii) $v_i = u_d$ and $u_i = v_d$ for $i\in \{d+1, \ldots, n\}$. The family of {\em Sierpi\'nski triangle graphs} $\ST_p^n$ was first introduced by Jakovac in~\cite{jakovac-2014}. These graphs can be defined in various ways, but for our purposes we do it for $p\ge 3$ as follows. If $p\ge 3$ and $n\ge 1$, then $\ST_p^n$ is the graph obtained from $S_p^n$ as follows. For any edge $uv$ which does not lie in a complete graph of order $p$, remove the edge $uv$ and identify the vertices $u$ and $v$. See Fig.~\ref{fig:S_3_3} where the graphs $S_3^3$ and $\ST_3^3$ are drawn.

\begin{figure}[ht!]
\begin{center}
\begin{tikzpicture}[scale=1.0,style=thick,x=1cm,y=1cm]
\def\vr{3pt}
\begin{scope}[xshift=0cm, yshift=0cm] 
\coordinate(x1) at (0.0,0.0);
\coordinate(x2) at (0.5,0.87);
\coordinate(x3) at (1,0);
\draw (x1) -- (x2) -- (x3) -- (x1);
\foreach \i in {1,...,3}
{
\draw(x\i)[fill=white] circle(\vr);
}
\node at (2,-1) {$\ST_3^3$};
\draw (0.5,0.3) ellipse (0.75cm and 0.75cm);
\end{scope}

\begin{scope}[xshift=1.0cm, yshift=0cm] 
\coordinate(x1) at (0.0,0.0);
\coordinate(x2) at (0.5,0.87);
\coordinate(x3) at (1,0);
\draw (x1) -- (x2) -- (x3) -- (x1);
\foreach \i in {1,...,3}
{
\draw(x\i)[fill=white] circle(\vr);
}
\draw (0.5,0.3) ellipse (0.75cm and 0.75cm);
\end{scope}

\begin{scope}[xshift=2.0cm, yshift=0cm] 
\coordinate(x1) at (0.0,0.0);
\coordinate(x2) at (0.5,0.87);
\coordinate(x3) at (1,0);
\draw (x1) -- (x2) -- (x3) -- (x1);
\foreach \i in {1,...,3}
{
\draw(x\i)[fill=white] circle(\vr);
}
\draw (0.5,0.3) ellipse (0.75cm and 0.75cm);
\end{scope}

\begin{scope}[xshift=3.0cm, yshift=0cm] 
\coordinate(x1) at (0.0,0.0);
\coordinate(x2) at (0.5,0.87);
\coordinate(x3) at (1,0);
\draw (x1) -- (x2) -- (x3) -- (x1);
\foreach \i in {1,...,3}
{
\draw(x\i)[fill=white] circle(\vr);
}
\draw (0.5,0.3) ellipse (0.75cm and 0.75cm);
\end{scope}

\begin{scope}[xshift=0.5cm, yshift=0.87cm] 
\coordinate(x1) at (0.0,0.0);
\coordinate(x2) at (0.5,0.87);
\coordinate(x3) at (1,0);
\draw (x1) -- (x2) -- (x3) -- (x1);
\foreach \i in {1,...,3}
{
\draw(x\i)[fill=white] circle(\vr);
}
\draw (0.5,0.3) ellipse (0.75cm and 0.75cm);
\end{scope}

\begin{scope}[xshift=2.5cm, yshift=0.87cm] 
\coordinate(x1) at (0.0,0.0);
\coordinate(x2) at (0.5,0.87);
\coordinate(x3) at (1,0);
\draw (x1) -- (x2) -- (x3) -- (x1);
\foreach \i in {1,...,3}
{
\draw(x\i)[fill=white] circle(\vr);
}
\draw (0.5,0.3) ellipse (0.75cm and 0.75cm);
\end{scope}

\begin{scope}[xshift=1.0cm, yshift=1.74cm] 
\coordinate(x1) at (0.0,0.0);
\coordinate(x2) at (0.5,0.87);
\coordinate(x3) at (1,0);
\draw (x1) -- (x2) -- (x3) -- (x1);
\foreach \i in {1,...,3}
{
\draw(x\i)[fill=white] circle(\vr);
}
\draw (0.5,0.3) ellipse (0.75cm and 0.75cm);
\end{scope}

\begin{scope}[xshift=2.0cm, yshift=1.74cm] 
\coordinate(x1) at (0.0,0.0);
\coordinate(x2) at (0.5,0.87);
\coordinate(x3) at (1,0);
\draw (x1) -- (x2) -- (x3) -- (x1);
\foreach \i in {1,...,3}
{
\draw(x\i)[fill=white] circle(\vr);
}
\draw (0.5,0.3) ellipse (0.75cm and 0.75cm);
\end{scope}

\begin{scope}[xshift=1.5cm, yshift=2.61cm] 
\coordinate(x1) at (0.0,0.0);
\coordinate(x2) at (0.5,0.87);
\coordinate(x3) at (1,0);
\draw (x1) -- (x2) -- (x3) -- (x1);
\foreach \i in {1,...,3}
{
\draw(x\i)[fill=white] circle(\vr);
}
\draw (0.5,0.3) ellipse (0.75cm and 0.75cm);
\end{scope}
\end{tikzpicture}
\hspace*{1cm}
\begin{tikzpicture}[scale=1.0,style=thick,x=1cm,y=1cm]
\def\vr{3pt}
\begin{scope}[xshift=0cm, yshift=0cm] 
\coordinate(x1) at (0.0,0.0);
\coordinate(x2) at (0.5,0.87);
\coordinate(x3) at (1,0);
\draw (x1) -- (x2) -- (x3) -- (x1);
\draw (1.0,0.0) -- (2.0,0.0);
\draw (3.0,0.0) -- (4.0,0.0);
\draw (5.0,0.0) -- (6.0,0.0);
\draw (3.0,3.48) -- (4.0,3.48);
\draw (0.5,0.87) -- (1.0,1.74);
\draw (1.5,2.61) -- (2.0,3.48);
\draw (2.5,4.35) -- (3.0,5.22);
\draw (6.5,0.87) -- (6.0,1.74);
\draw (5.5,2.61) -- (5.0,3.48);
\draw (4.5,4.35) -- (4.0,5.22);
\draw (2.5,0.87) -- (2.0,1.74);
\draw (4.5,0.87) -- (5.0,1.74);

\foreach \i in {1,...,3}
{
\draw(x\i)[fill=white] circle(\vr);
}
\node at (3.5,-1) {$S_3^3$};
\draw (0.5,0.3) ellipse (0.75cm and 0.75cm);
\end{scope}

\begin{scope}[xshift=2.0cm, yshift=0cm] 
\coordinate(x1) at (0.0,0.0);
\coordinate(x2) at (0.5,0.87);
\coordinate(x3) at (1,0);
\draw (x1) -- (x2) -- (x3) -- (x1);
\foreach \i in {1,...,3}
{
\draw(x\i)[fill=white] circle(\vr);
}
\draw (0.5,0.3) ellipse (0.75cm and 0.75cm);
\end{scope}

\begin{scope}[xshift=4.0cm, yshift=0cm] 
\coordinate(x1) at (0.0,0.0);
\coordinate(x2) at (0.5,0.87);
\coordinate(x3) at (1,0);
\draw (x1) -- (x2) -- (x3) -- (x1);
\foreach \i in {1,...,3}
{
\draw(x\i)[fill=white] circle(\vr);
}
\draw (0.5,0.3) ellipse (0.75cm and 0.75cm);
\end{scope}

\begin{scope}[xshift=6.0cm, yshift=0cm] 
\coordinate(x1) at (0.0,0.0);
\coordinate(x2) at (0.5,0.87);
\coordinate(x3) at (1,0);
\draw (x1) -- (x2) -- (x3) -- (x1);
\foreach \i in {1,...,3}
{
\draw(x\i)[fill=white] circle(\vr);
}
\draw (0.5,0.3) ellipse (0.75cm and 0.75cm);
\end{scope}

\begin{scope}[xshift=1.0cm, yshift=1.74cm] 
\coordinate(x1) at (0.0,0.0);
\coordinate(x2) at (0.5,0.87);
\coordinate(x3) at (1,0);
\draw (x1) -- (x2) -- (x3) -- (x1);
\foreach \i in {1,...,3}
{
\draw(x\i)[fill=white] circle(\vr);
}
\draw (0.5,0.3) ellipse (0.75cm and 0.75cm);
\end{scope}

\begin{scope}[xshift=5cm, yshift=1.74cm] 
\coordinate(x1) at (0.0,0.0);
\coordinate(x2) at (0.5,0.87);
\coordinate(x3) at (1,0);
\draw (x1) -- (x2) -- (x3) -- (x1);
\foreach \i in {1,...,3}
{
\draw(x\i)[fill=white] circle(\vr);
}
\draw (0.5,0.3) ellipse (0.75cm and 0.75cm);
\end{scope}

\begin{scope}[xshift=2.0cm, yshift=3.48cm] 
\coordinate(x1) at (0.0,0.0);
\coordinate(x2) at (0.5,0.87);
\coordinate(x3) at (1,0);
\draw (x1) -- (x2) -- (x3) -- (x1);
\foreach \i in {1,...,3}
{
\draw(x\i)[fill=white] circle(\vr);
}
\draw (0.5,0.3) ellipse (0.75cm and 0.75cm);
\end{scope}

\begin{scope}[xshift=4.0cm, yshift=3.48cm] 
\coordinate(x1) at (0.0,0.0);
\coordinate(x2) at (0.5,0.87);
\coordinate(x3) at (1,0);
\draw (x1) -- (x2) -- (x3) -- (x1);
\foreach \i in {1,...,3}
{
\draw(x\i)[fill=white] circle(\vr);
}
\draw (0.5,0.3) ellipse (0.75cm and 0.75cm);
\end{scope}

\begin{scope}[xshift=3cm, yshift=5.22cm] 
\coordinate(x1) at (0.0,0.0);
\coordinate(x2) at (0.5,0.87);
\coordinate(x3) at (1,0);
\draw (x1) -- (x2) -- (x3) -- (x1);
\foreach \i in {1,...,3}
{
\draw(x\i)[fill=white] circle(\vr);
}
\draw (0.5,0.3) ellipse (0.75cm and 0.75cm);
\end{scope}

\end{tikzpicture}
\caption{The graph $\ST_3^3$ and its sparse edge clique cover (left), and the graph $S_3^3$ (right)}
\label{fig:S_3_3}
\end{center}
\end{figure}
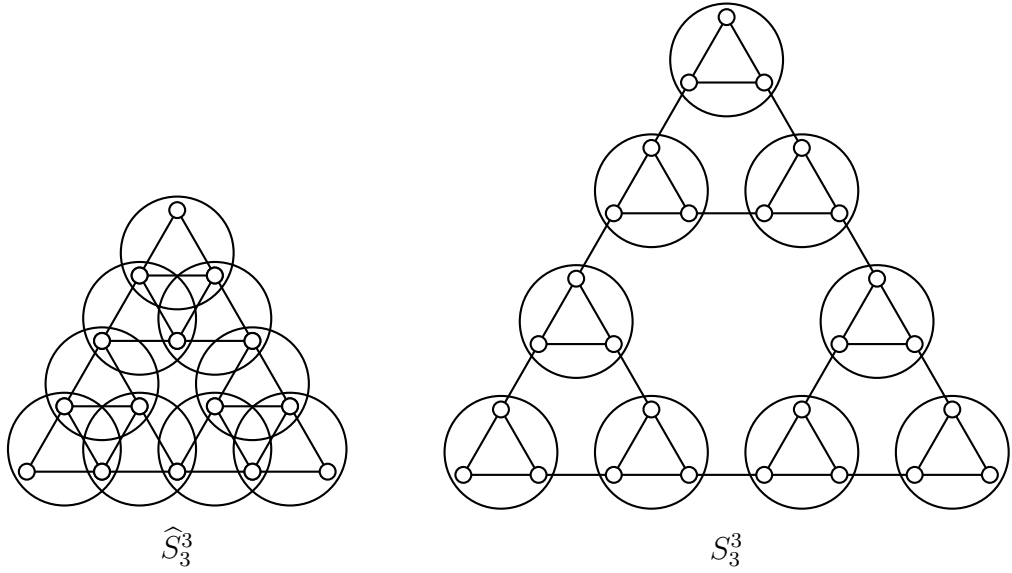

\begin{theorem}
\label{thm:sierpinski}
If $p\ge 3$ and $n\ge 1$, then $\chi_i(S_p^n)=p$. Moreover, $S_p^n$ is perfect injectively colorable.
\end{theorem}

\begin{proof}
As $\Delta(S_p^n)=p$, we have $\chi_i(S_p^n)\ge p$.

Let ${\cal C}$ be the edge clique cover of $\ST_p^n$ consisting of all the cliques of order $p$ of $\ST_p^n$ which are obtained from the cliques of order $p$ of $S_p^n$ after contacting the edges of $S_p^n$ which lie in no such clique. See Fig.~\ref{fig:S_3_3} where the described clique-edge cover of $\ST_3^3$ is shown. By the way this cover is constructed, we infer that ${\cal C}$ is a sparse clique-edge cover. Hence we can consider $(\ST_p^n)^{\cal C}$ and again by the construction we see that $(\ST_p^n)^{\cal C} \cong S_p^n$, see Fig.~\ref{fig:S_3_3} again. Applying Lemma~\ref{lem:vertexcoloring} we then get
$$\chi_i(S_p^n) = \chi_i((\ST_p^n)^{\cal C}) \le \chi(\ST_p^n) = p\,,$$
where the last equality is a result due to Jakovac proved in~\cite{jakovac-2014}. This proves the first assertion of the theorem.

To prove the second assertion, note that the vertex set of $S_p^n$ partitions into $p^{n-1}$ cliques of cardinality $p$. Combining this with $\chi_i(S_p^n)=p$, we infer that each color class of any $\chi_i(S_p^n)$-coloring has a non-empty intersection with each such clique of $S_p^n$. Since each color class is an open packing, we infer $\opack(S_p^n)\ge p^{n-1}$. On the other hand, note that if $P$ is an open packing of a graph, then no two vertices of $P$ can lie in a triangle in $G$. Therefore, when $p\ge 3$, we have $\opack(S_p^n)\le p^{n-1}$. Hence, $\opack(S_p^n)= p^{n-1}$, and moreover, every color class of any $\chi_i(S_p^n)$-coloring has $\opack(S_p^n)$ vertices. This proves the second assertion.
\end{proof}

We next give a short proof of the following result which was proved first by Hinz and Parisse~\cite{HP-2012} by a lengthy argument.

\begin{corollary}
If $p\ge 2$ and $n\ge 2$, then $S_p^n$ is a Class $1$ graph.  \end{corollary}

\begin{proof}
The case $p=2$ is clear since $S_2^n \cong P_{2^n}$. So, assume in the rest that $p\ge 3$. By Theorem~\ref{thm:sierpinski} we have $\chi_i(S_p^n)=p=\Delta(S_p^n)$, which by Lemma~\ref{lem:edgecoloring} immediately gives the conclusion.
\end{proof}

Gravier, Kov\v se, and Parreau~\cite{gravier-2011} defined generalized Sierpi\'nski graphs $S_p^n$ as follows. Let $G$ be an arbitrary graph. Then the {\em generalized Sierpi\'nski graph} $S_G^n$ is the graph with the vertex set $V(G)^n$, where two vertices $(u_1,\ldots, u_n)$ and $(v_1,\ldots, v_n)$ are adjacent if there exists an $i\in [n]$ such that $u_j = v_j$ for $j<i$, $u_iv_i\in E(G)$, and $u_j = v_i$ and $v_j = u_i$ for $j>i$. See~\cite{korze-2019} where the packing coloring of generalized Sierpi\'nski graphs was investigated.

Clearly, $\chi_i(S_G^n)\le |V(G)|$ follows from Theorem~\ref{thm:sierpinski} and the fact that $S_G^n$ is a spanning subgraph of $S_p^n$, where $p=|V(G)|$. The next result shows that this bound need not be sharp. In other words, Theorem~\ref{thm:sierpinski} does not have a counterpart in generalized Sierpi\'nski graphs.

\begin{proposition}
\label{prop:S_C_4}
If $n\ge 2$, then $\chi_i(S_{C_4}^n) = 3$ and $S_{C_4}^n$ is not perfect injectively colorable.
\end{proposition}

\begin{proof}
Let $n\ge 2$. Then $\Delta(S_{C_4}^n) = 3$, which implies that $\chi_i(S_{C_4}^n) \ge 3$. To see that $\chi_i(S_{C_4}^n) \le 3$ consider the labeling of $S_{C_4}^2$ as presented in the left-hand side of Fig.~\ref{fig:S_C_4}.

\begin{figure}[ht!]
\begin{center}
\begin{tikzpicture}[scale=1.0,style=thick,x=1cm,y=1cm]
\def\vr{3pt}

\begin{scope}[xshift=-5cm, yshift=2cm] 
\coordinate(x1) at (0,0);
\coordinate(x2) at (1,0);
\coordinate(x3) at (2,0);
\coordinate(x4) at (3,0);
\coordinate(x5) at (0,1);
\coordinate(x6) at (1,1);
\coordinate(x7) at (2,1);
\coordinate(x8) at (3,1);
\coordinate(x9) at (0,2);
\coordinate(x10) at (1,2);
\coordinate(x11) at (2,2);
\coordinate(x12) at (3,2);
\coordinate(x13) at (0,3);
\coordinate(x14) at (1,3);
\coordinate(x15) at (2,3);
\coordinate(x16) at (3,3);
\draw (x1) -- (x2) -- (x6) -- (x5) -- (x1);
\draw (x2) -- (x3) -- (x4) -- (x8) -- (x7) -- (x3);
\draw (x8) -- (x12) -- (x16) -- (x15) -- (x11) -- (x12);
\draw (x15) -- (x14) -- (x13) -- (x9) -- (x10) -- (x14);
\draw (x5) -- (x9);
\foreach \i in {1,...,16}
{
\draw(x\i)[fill=white] circle(\vr);
}
\draw[below] (x1)++(0,-0.1) node {$3$};
\draw[below] (x2)++(0,-0.1) node {$1$};
\draw[below] (x3)++(0,-0.1) node {$2$};
\draw[below] (x4)++(0,-0.1) node {$3$};
\draw[above] (x13)++(0,0.1) node {$3$};
\draw[above] (x14)++(0,0.1) node {$1$};
\draw[above] (x15)++(0,0.1) node {$2$};
\draw[above] (x16)++(0,0.1) node {$3$};
\draw[left] (x5)++(-0.1,0) node {$2$};
\draw[left] (x9)++(-0.1,0) node {$2$};
\draw[right] (x8)++(0.1,0) node {$1$};
\draw[right] (x12)++(0.1,0) node {$1$};
\draw[right] (x10)++(0.1,0) node {$1$};
\draw[below] (x11)++(0,-0.1) node {$2$};
\draw[left] (x7)++(-0.1,0) node {$2$};
\draw[above] (x6)++(0,0.1) node {$1$};
\end{scope}

\begin{scope}[xshift=0cm, yshift=0cm] 
\coordinate(x1) at (0,0);
\coordinate(x2) at (1,0);
\coordinate(x3) at (2,0);
\coordinate(x4) at (3,0);
\coordinate(x5) at (0,1);
\coordinate(x6) at (1,1);
\coordinate(x7) at (2,1);
\coordinate(x8) at (3,1);
\coordinate(x9) at (0,2);
\coordinate(x10) at (1,2);
\coordinate(x11) at (2,2);
\coordinate(x12) at (3,2);
\coordinate(x13) at (0,3);
\coordinate(x14) at (1,3);
\coordinate(x15) at (2,3);
\coordinate(x16) at (3,3);
\draw (x1) -- (x2) -- (x6) -- (x5) -- (x1);
\draw (x2) -- (x3) -- (x4) -- (x8) -- (x7) -- (x3);
\draw (x8) -- (x12) -- (x16) -- (x15) -- (x11) -- (x12);
\draw (x15) -- (x14) -- (x13) -- (x9) -- (x10) -- (x14);
\draw (x5) -- (x9);
\draw (3,0) -- (4,0);
\draw (0,3) -- (0,4);
\draw (7,3) -- (7,4);
\draw (3,7) -- (4,7);
\foreach \i in {1,...,16}
{
\draw(x\i)[fill=white] circle(\vr);
}
\draw[left] (x13)++(-0.1,0) node {$3$};
\draw[below] (x4)++(0,-0.1) node {$3$};
\draw[below] (x1)++(0,-0.1) node {$3$};
\draw[above] (x16)++(0,0.1) node {$3$};
\end{scope}

\begin{scope}[xshift=4cm, yshift=0cm] 
\coordinate(x1) at (0,0);
\coordinate(x2) at (1,0);
\coordinate(x3) at (2,0);
\coordinate(x4) at (3,0);
\coordinate(x5) at (0,1);
\coordinate(x6) at (1,1);
\coordinate(x7) at (2,1);
\coordinate(x8) at (3,1);
\coordinate(x9) at (0,2);
\coordinate(x10) at (1,2);
\coordinate(x11) at (2,2);
\coordinate(x12) at (3,2);
\coordinate(x13) at (0,3);
\coordinate(x14) at (1,3);
\coordinate(x15) at (2,3);
\coordinate(x16) at (3,3);
\draw (x1) -- (x2) -- (x6) -- (x5) -- (x1);
\draw (x2) -- (x3) -- (x4) -- (x8) -- (x7) -- (x3);
\draw (x8) -- (x12) -- (x16) -- (x15) -- (x11) -- (x12);
\draw (x15) -- (x14) -- (x13) -- (x9) -- (x10) -- (x14);
\draw (x5) -- (x9);
\foreach \i in {1,...,16}
{
\draw(x\i)[fill=white] circle(\vr);
}
\draw[below] (x1)++(0,-0.1) node {$3$};
\draw[right] (x16)++(0.1,0) node {$3$};
\draw[below] (x4)++(0,-0.1) node {$3$};
\draw[left] (x13)++(-0.1,0) node {$3$};
\end{scope}

\begin{scope}[xshift=0cm, yshift=4cm] 
\coordinate(x1) at (0,0);
\coordinate(x2) at (1,0);
\coordinate(x3) at (2,0);
\coordinate(x4) at (3,0);
\coordinate(x5) at (0,1);
\coordinate(x6) at (1,1);
\coordinate(x7) at (2,1);
\coordinate(x8) at (3,1);
\coordinate(x9) at (0,2);
\coordinate(x10) at (1,2);
\coordinate(x11) at (2,2);
\coordinate(x12) at (3,2);
\coordinate(x13) at (0,3);
\coordinate(x14) at (1,3);
\coordinate(x15) at (2,3);
\coordinate(x16) at (3,3);
\draw (x1) -- (x2) -- (x6) -- (x5) -- (x1);
\draw (x2) -- (x3) -- (x4) -- (x8) -- (x7) -- (x3);
\draw (x8) -- (x12) -- (x16) -- (x15) -- (x11) -- (x12);
\draw (x15) -- (x14) -- (x13) -- (x9) -- (x10) -- (x14);
\draw (x5) -- (x9);
\foreach \i in {1,...,16}
{
\draw(x\i)[fill=white] circle(\vr);
}
\draw[left] (x1)++(-0.1,0) node {$3$};
\draw[above] (x16)++(0,0.1) node {$3$};
\draw[right] (x4)++(0.1,0) node {$3$};
\draw[above] (x13)++(0,0.1) node {$3$};
\end{scope}

\begin{scope}[xshift=4cm, yshift=4cm] 
\coordinate(x1) at (0,0);
\coordinate(x2) at (1,0);
\coordinate(x3) at (2,0);
\coordinate(x4) at (3,0);
\coordinate(x5) at (0,1);
\coordinate(x6) at (1,1);
\coordinate(x7) at (2,1);
\coordinate(x8) at (3,1);
\coordinate(x9) at (0,2);
\coordinate(x10) at (1,2);
\coordinate(x11) at (2,2);
\coordinate(x12) at (3,2);
\coordinate(x13) at (0,3);
\coordinate(x14) at (1,3);
\coordinate(x15) at (2,3);
\coordinate(x16) at (3,3);
\draw (x1) -- (x2) -- (x6) -- (x5) -- (x1);
\draw (x2) -- (x3) -- (x4) -- (x8) -- (x7) -- (x3);
\draw (x8) -- (x12) -- (x16) -- (x15) -- (x11) -- (x12);
\draw (x15) -- (x14) -- (x13) -- (x9) -- (x10) -- (x14);
\draw (x5) -- (x9);
\foreach \i in {1,...,16}
{
\draw(x\i)[fill=white] circle(\vr);
}
\draw[right] (x4)++(0.1,0) node {$3$};
\draw[above] (x13)++(0,0.1) node {$3$};
\draw[below] (x1)++(0,-0.1) node {$3$};
\draw[above] (x16)++(0,0.1) node {$3$};
\end{scope}

\end{tikzpicture}
\caption{An injective coloring of $S_{C_4}^2$ (left) and its lift up to an injective coloring of $S_{C_4}^3$ (right)}
\label{fig:S_C_4}
\end{center}
\end{figure}
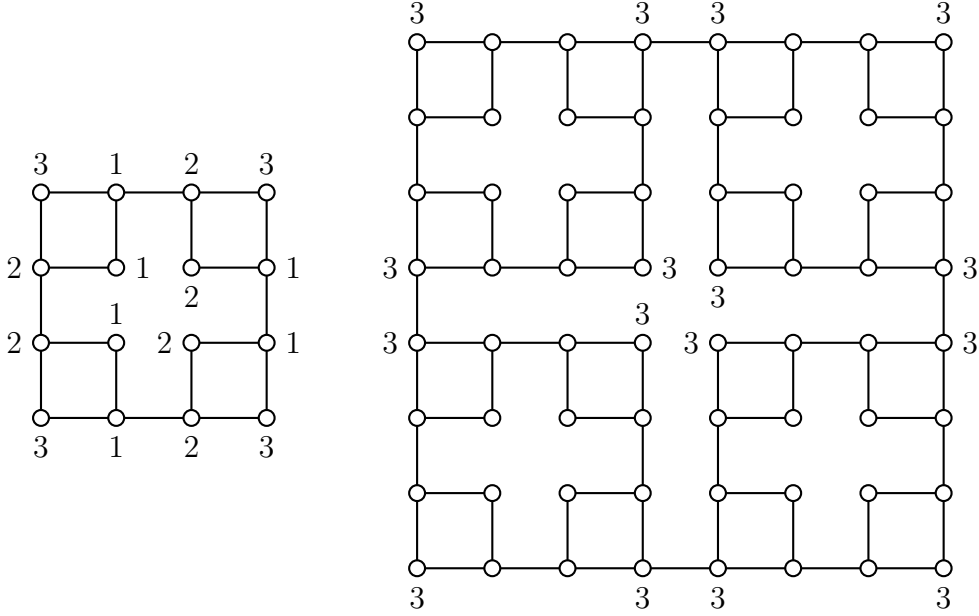

The labeling of $S_{C_4}^2$ from Fig.~\ref{fig:S_C_4} is easily checked to be injective. Now we iteratively proceed as indicated in the figure, that is, we four times use the labeling of $S_{C_4}^2$ to get a labeling of $S_{C_4}^3$. Based on the distribution of the color $3$, it is straightforward to check that also the labeling of $S_{C_4}^3$ is injective. The process can be repeated to get the desired conclusion.

Finally, notice that $S_{C_4}^n$ is not perfect injectively colorable because $|V(S_{C_4}^n)|=4^n$, which is not divisible by $\chi_i(S_{C_4}^n)=3$ for $n\ge 2$.
\end{proof}

Using similar approach as in the proof of Proposition~\ref{prop:S_C_4} one might deduce that if $n\ge 2$ and $k\ge 5$, then $\chi_i(S_{C_k}^n) = 3$. Moreover, one might also deduce that among all these generalized Sierpi\'nski graphs over cycles only $S_{C_6}^n$ is perfect injectively colorable.

\section{Rooted product graphs}
\label{sec:rooted}

A \emph{rooted graph} is a graph in which one vertex is labeled in a special way to distinguish it from other vertices. This vertex is called the \emph{root} of the graph. Let $G$ be a graph with vertex set $\{v_1,\ldots,v_n\}$. Let ${\cal H}$ be a sequence of $n$ rooted graphs $H_1,\ldots,H_n$. The \emph{rooted product graph} $G({\cal H})$ is the graph obtained by identifying the root of $H_i$ with $v_{i}$ (see \cite{GM}). We here consider the particular case of rooted product graphs where ${\cal H}$ consists of $n$  isomorphic rooted graphs \cite{Sch}. More formally, assuming that the root of $H$ is $v$, we define the rooted product graph $G\circ_{v}H=(V,E)$, where $V=V(G)\times V(H)$ and
$$E=\left(\bigcup_{i=1}^{n}\left\{(v_i,h)(v_i,h')\mid hh'\in E(H)\right\}\right)\bigcup \Big{\{}(v_i,v)(v_j,v)\mid v_iv_j\in E(G)\Big{\}}.$$

We remark that rooted product graphs can be seen as an instance of the operation called Sierpi\'nski product introduced in \cite{kpzz-2022}, and denoted by $G \otimes _f H$, where $f \colon V(G)\rightarrow V(H)$ is a function. The graph $G \otimes _f H$ has vertex set $V(G)\times V(H)$ and two vertices $(g,h),(g',h')\in V(G \otimes _f H)$ are adjacent if (i) $g=g'$ and $hh' \in E(H)$, or (ii)
$gg' \in E(G)$, $h=f(g')$ and $h'=f(g)$. In this sense, it can be readily seen that a rooted product graph $G\circ_{v}H$ represents a Sierpi\'nski product $G\otimes _f H$,  where $f$ is a constant function in the product, i.e., $f(u)=v$ for any $u\in V(G)$.

It is somehow natural to think that the injective chromatic number of a rooted product graph relates to that of the factors of the product. Indeed, the following basic bounds can be easily deduced for any graph $G$ and any rooted graph $H$ with root $v$.
\begin{equation}
\label{eq:trivial-bounds-rooted}
\max\{\chi_i(G),\chi_i(H)\}\le \chi_i(G\circ_v H)\le \chi_i(G)+\chi_i(H).
\end{equation}

Both bounds above are realizable. However, not all possible values between these bounds are reached. We next focus on these facts and show that $\chi_i(G\circ_v H)$ achieves only six values from the interval $[\max\{\chi_i(G),\chi_i(H)\},\dots, \chi_i(G)+\chi_i(H)]$.

\bigskip
From now on, in order to facilitate our exposition, given an integer $k\in [n]$, by $F_{k}$ we represent the subgraph of $G\circ_{v} H$ induced by vertices in $V(G)\cup V(H_{k})$.

To show that only six values from the interval $[\max\{\chi_i(G),\chi_i(H)\},\dots, \chi_i(G)+\chi_i(H)]$ (according to the bounds from \eqref{eq:trivial-bounds-rooted}) can be realized, we exhibit a closed formula for the injective chromatic number of rooted product graphs. To do so, we proceed with a series of lemmas that are giving the values of the injective chromatic number of $F_{k}$, under some conditions happening in $G$ and in $H$. For the sake of convenience, by assigning/giving a color to a vertex subset $S$ of a graph $G$ we mean assigning such color to all vertices in $S$.

We first remark that the root $v\in V(H)$ is identified with $v_{k}\in V(G)$. Hence, for instance $d_{G}(v)=d_{G}(v_{k})$. Throughout the remainder of this section, we consider $g$ as a $\chi_{i}(G)$-coloring, and in this sense, $\{U_{1},\ldots,U_{\chi_{i}(G)}\}$ as the vertex partition of $V(G)$ into open packings associated with $g$. We may assume that $g$ assigns the color $i$ to $U_{i}$ for each $i\in[\chi_{i}(G)]$. Now, in order to consider an injective coloring of $F_k$ for each $k\in [n]$, in concordance with the $\chi_{i}(G)$-coloring $g$, we may also assume that $N_{G}(v)\subseteq U_{1}\cup\ldots \cup U_{d_{G}(v)}$. Moreover, by simplicity we write $H=H_{k}$ in the proofs of the following lemmas, as $H_{k}\cong H$.

We first observe that such subgraph $F_k$ satisfies that
\begin{equation}\label{max}
\chi_{i}(F_{k})\geq \max\big{\{}\chi_{i}(G),\chi_{i}(H),d_{G}(v)+d_{H}(v)\big{\}}
\end{equation}
as both $G$ and $H$ are subgraphs of $F_{k}$ and $\chi_{i}(F_{k})\geq \Delta(F_{k})\geq d_{G}(v)+d_{H}(v)$. The proof of this is based on a case-by-case analysis. The general procedure is to extend $g$ to $F_{k}$ in such a way that the restriction of the resulting function to $V(H)$ turns out to be a $\chi_{i}(H)$-coloring using the colors in $[\chi_{i}(G)]$ along with the least number of colors not in $[\chi_{i}(G)]$.

\medskip
Our first lemma regarding the injective chromatic number of the graph $F_k$ reads as follows.

\begin{lemma}\label{L1}
If $g(v)\notin[d_{G}(v)]$ and $h(v)\in h\big{(}N_{H}(v)\big{)}$ for some $\chi_{i}(H)$-function $h=(V_{1},\ldots,V_{\chi_{i}(H)})$, then
$$\chi_{i}(F_{k})=\max\big{\{}\chi_{i}(G),\chi_{i}(H),d_{G}(v)+d_{H}(v)\big{\}}.$$
\end{lemma}

\begin{proof}
We first observe that $\chi_{i}(G)\geq d_{G}(v)+1$ since $g(v)\notin[d_{G}(v)]$ (i.e., $v$ and all its neighbors have different colors). We need to distinguish two cases depending on the behaviors of $\chi_{i}(G)$ and $\chi_{i}(H)$. Without loss of generality, for the graph $H=H_k$, assume that $N_{H}(v)\subseteq V_{1}\cup\ldots \cup V_{d_{H}(v)}$ and that $v\in V_{1}$. Also, for the graph $G$, assume that $v\in U_{d_{G}(v)+1}$.

\medskip
\noindent
\textbf{Case 1.} $\chi_{i}(G)\geq d_{H}(v)+d_{G}(v)$. \\
If $\chi_{i}(H)=d_{H}(v)$, then we extend $g$ to $F_{k}$ by  respectively assigning the colors $d_{G}(v)+1,\ldots,d_{G}(v)+d_{H}(v)$ to $V_{1},\ldots,V_{d_{H}(v)}$. Note that the resulting function $g_{1}$ is an injective coloring of $F_{k}$ with $\chi_{i}(G)$ colors. Therefore, $\chi_{i}(F_{k})\le \chi_{i}(G)$, which is indeed an equality since $G$ is a subgraph of $F_{k}$. Now, if $\chi_{i}(H)>d_{H}(v)$, then we proceed as follows. If $\eta=\chi_{i}(G)-d_{G}(v)-d_{H}(v)\geq \chi_{i}(H)-d_{H}(v)=\varphi$, then we extend $g_{1}$ to a new function $g_{2}$, by respectively giving the colors $d_{G}(v)+d_{H}(v)+1,\ldots,d_{G}(v)+\chi_{i}(H)$ to $V_{d_{H}(v)+1},\ldots,V_{\chi_{i}(H)}$. It is readily seen that $g_{2}$ is an injective coloring of $F_{k}$ with $\chi_{i}(G)$ colors. This similarly results in the equality $\chi_{i}(F_{k})=\chi_{i}(G)$. Hence, assume $\eta<\varphi$ and consider two cases depending on $\eta$.

\smallskip
\noindent
\textbf{Subcase 1.1.} $\eta=0$.\\
If $\varphi\leq d_{G}(v)$, then we assign $\varphi$ colors from $[d_{G}(v)]$ to $V_{i}$ for $i=d_{H}(v)+1,\ldots,\chi_{i}(H)$. This gives us an injective coloring of $F_{k}$ using $\chi_{i}(G)$ colors, and hence $\chi_{i}(F_{k})=\chi_{i}(G)$. If $\varphi>d_{G}(v)$, then we assign the colors from $[d_{G}(v)]$ to $V_{i}$ with $i=d_{H}(v)+1,\ldots,d_{H}(v)+d_{G}(v)$. In addition, we need $\chi_{i}(H)-d_{G}(v)-d_{H}(v)$ new colors for the rest of open packings in $H$. This results in an injective coloring of $F_{k}$ with $\chi_{i}(H)-d_{G}(v)-d_{H}(v)+\chi_{i}(G)=\chi_{i}(H)$ colors, and therefore $\chi_{i}(F_{k})=\chi_{i}(H)$.

\smallskip
\noindent
\textbf{Subcase 1.2.} $\eta>0$. \\
Let $g_{2}$ be an extension of $g_1$ that respectively assigns the values $d_{G}(v)+d_{H}(v)+1,\ldots,\chi_{i}(G)$ to $V_{d_{H}(v)+1}, \ldots,  V_{\chi_{i}(G)-d_{G}(v)}$. In view of this, $V_{\chi_{i}(G)-d_{G}(v)+1},\ldots, V_{\chi_{i}(H)}$ are open packings of $H$ which have not been injectively colored by $g_{2}$. We need to consider two possibilities.

\smallskip
\noindent
\textbf{Subcase 1.2.1.} $\chi_{i}(G)\geq \chi_{i}(H)$. \\
It follows that $\zeta=\chi_{i}(H)-\chi_{i}(G)+d_{G}(v)\leq d_{G}(v)$. In such a situation, $g_{2}$ can be extended to $F_{k}$ by assigning the colors $1,\ldots,\zeta$ to the rest of open packings in $H$. Note that the resulting function is an injective coloring of $F_{k}$ with $\chi_{i}(G)$ colors, and hence $\chi_{i}(F_{k})=\chi_{i}(G)$.

\smallskip
\noindent
\textbf{Subcase 1.2.2.} $\chi_{i}(G)<\chi_{i}(H)$. \\
This shows that $\zeta>d_{G}(v)$. In such a situation, we first assign the color $i$ to $V_{\chi_{i}(G)-d_{G}(v)+i}$ for each $i\in[d_{G}(v)]$. We next assign $\chi_{i}(H)-\chi_{i}(G)$ new colors to the rest of open packings in $H$. This results in an injective coloring of $F_{k}$ with $\chi_{i}(H)$ colors, and hence $\chi_{i}(F_{k})=\chi_{i}(H)$.

\medskip
\noindent
\textbf{Case 2.} $\chi_{i}(G)<d_{H}(v)+d_{G}(v)$. \\
We extend $g$ to $g_{1}$ by respectively assigning the colors $d_{G}(v)+1,\ldots,\chi_{i}(G)$ to $V_{1},\ldots,V_{\chi_{i}(G)-d_{G}(v)}$, as well as, we assign $d_{H}(v)-\chi_{i}(G)+d_{G}(v)$ new colors to $V_{\chi_{i}(G)-d_{G}(v)+1},\ldots,V_{d_{H}(v)}$. If $\chi_{i}(H)=d_{H}(v)$, then $g_{1}$ turns out to be an injective coloring of $F_{k}$ with $d_{G}(v)+d_{H}(v)$ colors. Therefore, $\chi_{i}(F_{k})=d_{H}(v)+d_{G}(v)$. Suppose now that $\chi_{i}(H)>d_{H}(v)$. We distinguish two more possibilities.

\smallskip
\noindent
\textbf{Subcase 2.1.} $\chi_{i}(H)\le d_{G}(v)+d_{H}(v)$. \\
In such situation, the function $g_{1}$ can be extended to $F_{k}$ by respectively assigning the colors $1,\ldots,\chi_{i}(H)-d_{H}(v)$ to $V_{d_{H}(v)+1},\ldots,V_{\chi_{i}(H)}$. This gives us an injective coloring of $F_{k}$ with $d_{H}(v)+d_{G}(v)$ colors. So, $\chi_{i}(F_{k})=d_{H}(v)+d_{G}(v)$.

\smallskip
\noindent
\textbf{Subcase 2.2.} $\chi_{i}(H) > d_{G}(v)+d_{H}(v)$. \\
Now, in order to extend $g_{1}$ to $F_{k}$, we respectively assign the colors $1,\ldots,d_{G}(v)$ to $V_{d_{H}(v)+1},\ldots,V_{d_{H}(v)+d_{G}(v)}$, as well as, $\chi_{i}(H)-d_{G}(v)-d_{H}(v)$ new colors to the rest of open packings in $H$. This leads to an injective coloring of $F_{k}$ with $\chi_{i}(H)$ colors, and hence $\chi_{i}(F_{k})=\chi_{i}(H)$.

\bigskip
One might think that in order to complete our proof, some other cases like for instance $\chi_{i}(H)\leq d_{H}(v)+d_{G}(v)$ or $\chi_{i}(H) > d_{H}(v)+d_{G}(v)$ need to be considered. However, they are indeed implicitly checked in the two cases above (with the corresponding subcases).

\medskip
In conclusion, we have proved that $\chi_{i}(F_{k})\in \{\chi_{i}(G),\chi_{i}(H),d_{G}(v)+d_{H}(v)\}$. This leads to $\chi_{i}(F_{k})=\max\{\chi_{i}(G),\chi_{i}(H),d_{G}(v)+d_{H}(v)\}$ due to (\ref{max}).
\end{proof}

From this point on, three similar lemmas to the one above shall be proved. These lemmas are considering the remaining cases regarding the inclusion or not of $g(v)$ and $h(v)$ in $[d_{G}(v)]$ and $h\big{(}N_{H}(v)\big{)}$ (for some/each $\chi_{i}(H)$-function $h$), respectively. Some of the arguments are similar to the ones in the proof of Lemma \ref{L1}.

\begin{lemma}\label{L2}
If $g(v)\notin[d_{G}(v)]$ and $h(v)\notin h\big{(}N_{H}(v)\big{)}$ for each $\chi_{i}(H)$-function $h$, then
$$\chi_{i}(F_{k})\in \big{\{}\chi_{i}(G),\chi_{i}(H),d_{G}(v)+d_{H}(v),d_{G}(v)+d_{H}(v)+1\big{\}}.$$
\end{lemma}

\begin{proof}
The assumptions imply that $\chi_{i}(G)\geq d_{G}(v)+1$ and $\chi_{i}(H)\geq d_{H}(v)+1$. We need to distinguish two cases depending on $\chi_{i}(G)$, $d_{G}(v)$ and $d_{H}(v)$. For the sake of simplicity, we assume $N_{H}(v)\subseteq V_{1}\cup \ldots \cup V_{d_{H}(v)}$, $v\in V_{d_{H}(v)+1}$ and $v\in U_{d_{G}(v)+1}$.

\medskip
\noindent
\textbf{Case 1.} $\chi_{i}(G)\geq d_{G}(v)+d_{H}(v)+1$. \\
If $\chi_{i}(H)=d_{H}(v)+1$, then the extension $g_{1}$ of $g$ that respectively assigns the colors $d_{G}(v)+1,d_{G}(v)+2,\ldots,d_{G}(v)+d_{H}(v)+1$ to $V_{d_{H}(v)+1},V_{1},\ldots,V_{d_{H}(v)}$ defines an injective coloring of $F_{k}$ using $\chi_{i}(G)$ colors. This leads to $\chi_{i}(F_{k})=\chi_{i}(G)$. Suppose now that $\chi_{i}(H)>d_{H}(v)+1$. If
\begin{center}
$\mu=\chi_{i}(G)-d_{G}(v)-d_{H}(v)-1\geq \chi_{i}(H)-d_{H}(v)-1=\xi$,
\end{center}
then an extension $g_{2}$ of $g_{1}$ assigning the color $i$ to $V_{i-d_{G}(v)}$, for each $i=d_{G}(v)+d_{H}(v)+2,\ldots,d_{G}(v)+\chi_{i}(H)$, defines an injective coloring of $F$ with $\chi_{i}(G)$ colors. Hence, $\chi_{i}(F_{k})=\chi_{i}(G)$. Letting $\mu<\xi$, we need to consider two more possibilities.

\smallskip
\noindent
\textbf{Subcase 1.1.} $\mu=0$. \\
First note that the $\xi$ open packings $V_{d_{H}(v)+2},\ldots,V_{\chi_{i}(H)}$ of $H$ have not been colored under $g_{1}$. If $\xi\leq d_{G}(v)$, then we respectively assign the colors $1,\ldots,\xi$ to $V_{d_{H}(v)+2},\ldots,V_{\chi_{i}(H)}$. Note that the resulting function is an injective coloring with $\chi_{i}(G)$ colors. So, $\chi_{i}(F_{k})=\chi_{i}(G)$. If $\xi>d_{G}(v)$, then we first respectively assign the colors $1,\ldots,d_{G}(v)$ to $V_{d_{H}(v)+2},\ldots,V_{d_{H}(v)+d_{G}(v)+1}$. Next, we assign $\xi-d_{G}(v)$ new colors to the rest of the open packings of $H$. The resulting function turns out to be an injective coloring of $F_{k}$ using
\begin{center}
$\chi_{i}(G)+\xi-d_{G}(v)=\chi_{i}(G)+\chi_{i}(H)-d_{H}(v)-1-d_{G}(v)=\chi_{i}(H)$
\end{center}
colors. So, we deduce that $\chi_{i}(F_{k})\leq \chi_{i}(H)$, which means $\chi_{i}(F_{k})=\chi_{i}(H)$ in view of the inequality (\ref{max}).

\smallskip
\noindent
\textbf{Subcase 1.2.} $\mu>0$. \\
As an extension of $g_{1}$, we first assign the color $i$ to $V_{i-d_{G}(v)}$ when $i\in \{d_{G}(v)+d_{H}(v)+2,\ldots,\chi_{i}(G)\}$. In this situation, the $\xi-\mu$ open packings $V_{\chi_{i}(G)-d_{G}(v)+1},\ldots,V_{\chi_{i}(H)}$ have not been colored under $g_{1}$. We now distinguish two possibilities.

\smallskip
\noindent
\textbf{Subcase 1.2.1.} $\chi_{i}(G)\geq \chi_{i}(H)$. This implies that $\xi-\mu\leq d_{G}(v)$. So, by respectively assigning the colors $1,\ldots,\xi-\mu$ to $V_{\chi_{i}(G)-d_{G}(v)+1},\ldots,V_{\chi_{i}(H)}$, we obtain an injective coloring of $F_{k}$ with $\chi_{i}(G)$ colors. Therefore, $\chi_{i}(F_{k})=\chi_{i}(G)$.

\smallskip
\noindent
\textbf{Subcase 1.2.2.} $\chi_{i}(G)<\chi_{i}(H)$. \\
This shows that $\xi-\mu>d_{G}(v)$. In this situation, we respectively assign the values $1,\ldots,d_{G}(v)$ to $V_{\chi_{i}(G)-d_{G}(v)+1},\ldots,V_{\chi_{i}(G)}$. Also, we assign $\chi_{i}(H)-\chi_{i}(G)$ new colors to the rest of open packings in $H$. Note that the resulting coloring of $F_{k}$ is injective, and that it uses $\chi_{i}(H)$ colors. Hence, $\chi_{i}(F_{k})=\chi_{i}(H)$.

\medskip
\noindent
\textbf{Case 2.} $\chi_{i}(G)<d_{H}(v)+d_{G}(v)+1$. \\
Consider first that $\chi_{i}(H)=d_{H}(v)+1$. In such a situation, let $g_{1}$ be an extension of $g$ that respectively assigns the colors $d_{G}(v)+1,d_{G}(v)+2,\ldots,\chi_{i}(G)$ to $V_{d_{H}(v)+1},V_{1},\ldots,V_{\chi_{i}(G)-d_{G}(v)-1}$, as well as, $\varphi=\chi_{i}(H)-\chi_{i}(G)+d_{G}(v)$ new colors to $V_{\chi_{i}(G)-d_{G}(v)},\ldots,V_{\chi_{i}(H)}$. This defines an injective coloring of $F_{k}$ with $\varphi+\chi_{i}(G)=d_{G}(v)+d_{H}(v)+1$ colors, and hence $\chi_{i}(F_{k})\leq d_{G}(v)+d_{H}(v)+1$. This shows that $\chi_{i}(F)\in \{d_{G}(v)+d_{H}(v),d_{G}(v)+d_{H}(v)+1\}$ due to (\ref{max}). Suppose now that $\chi_{i}(H)>d_{H}(v)+1$. We need to consider two possibilities.

\smallskip
\noindent
\textbf{Subcase 2.1.} $\chi_{i}(G)\geq \chi_{i}(H)$. \\
Due to the initial inequality of Case 2, we get $d_{G}(v)+d_{H}(v)+1>\chi_{i}(H)$. Therefore, $d_{G}(v)>\chi_{i}(H)-d_{H}(v)-1=\psi$. Hence, $g_{1}$ can be extended to $F_{k}$ by respectively assigning the colors $1,\ldots,\psi$ to $V_{d_{H}(v)+2},\ldots,V_{\chi_{i}(H)}$. This gives an injective coloring of $F_{k}$ with $d_{G}(v)+d_{H}(v)+1$ colors. Consequently, $\chi_{i}(F_{k})\in \{d_{G}(v)+d_{H}(v),d_{G}(v)+d_{H}(v)+1\}$.

\smallskip
\noindent
\textbf{Subcase 2.2.} $\chi_{i}(G)<\chi_{i}(H)$. \\
If $\psi\leq d_{G}(v)$, then we have the same conclusion as in Subcase 2.1. So, let $\psi>d_{G}(v)$. Let $g_{2}$ be an extension of $g_{1}$ that assigns the color $i$ to $V_{i+d_{H}(v)+1}$ for each $i\in[d_{G}(v)]$. In such a situation, we give $\psi-d_{G}(v)$ new colors to the rest of the open packings in $H$. This process leads to an injective coloring of $F_{k}$ with at most $d_{G}(v)+d_{H}(v)+1+\psi-d_{G}(v)$ colors. Therefore, $\chi_{i}(F_{k})\leq \chi_{i}(H)$. This implies that $\chi_{i}(F_{k})=\chi_{i}(H)$.
\end{proof}

\begin{lemma}\label{L3}
Let $g(v)\in[d_{G}(v)]$ and $h(v)\in h\big{(}N_{H}(v)\big{)}$ for some $\chi_{i}(H)$-function $h=(V_{1},\ldots,V_{\chi_{i}(H)})$. Then,
$$\chi_{i}(F_{k})\in \big{\{}\chi_{i}(G),\chi_{i}(H),\chi_{i}(G)+1,\chi_{i}(H)+1,d_{G}(v)+d_{H}(v)\big{\}}.$$
\end{lemma}

\begin{proof}
With the assumptions given in the statement of the lemma, we may assume that $N_{H}(v)\subseteq V_{1}\cup\ldots \cup V_{d_{H}(v)}$, $v\in V_{1}$ and $v\in U_{1}$. We distinguish two cases depending on $\chi_{i}(G)$, $d_{G}(v)$ and $d_{H}(v)$.

\medskip
\noindent
\textbf{Case 1.} $\chi_{i}(G)\geq d_{H}(v)+d_{G}(v)$. \\
Let $\chi_{i}(H)=d_{H}(v)$. In such a situation, assume $g_1$   respectively assigns the colors $d_{G}(v)+1,\ldots,d_{G}(v)+d_{H}(v)$ to $V_{1}\setminus \{v\},V_{2},\ldots,V_{d_{H}(v)}$. This results in the existence of an injective coloring of $F_{k}$ with $\chi_{i}(G)$ colors, and hence $\chi_{i}(F_{k})=\chi_{i}(G)$. Now let $\chi_{i}(H)>d_{H}(v)$. If $\vartheta=\chi_{i}(G)-d_{G}(v)-d_{H}(v)\geq \chi_{i}(H)-d_{H}(v)=\epsilon$, then we consider $g_{1}$ is extended to $g_{2}$ by assigning $i$ to $V_{i-d_{G}(v)}$ for each $i=d_{G}(v)+d_{H}(v)+1,\ldots,d_{G}(v)+\chi_{i}(H)$. This defines an injective coloring of $F_{k}$ using $\chi_{i}(G)$ colors. Hence, $\chi_{i}(F_{k})=\chi_{i}(G)$. Letting $\vartheta<\epsilon$ we need to consider two more cases depending on $\vartheta$.

\smallskip
\noindent
\textbf{Subcase 1.1.} $\vartheta=0$. \\
If $\epsilon\leq d_{G}(v)-1$, then in order to extend $g_{1}$, we respectively assign the colors $2,\ldots,\epsilon+1$ to $V_{d_{H}(v)+1},\ldots,V_{\chi_{i}(H)}$. Note that the resulting function is an injective coloring of $F_{k}$ using $\chi_{i}(G)$ colors, and hence $\chi_{i}(F_{k})=\chi_{i}(G)$. Suppose now that $\epsilon\geq d_{G}(v)$. As an extension of $g_1$, we first respectively assign the colors $2,\ldots,d_{G}(v)$ to $V_{d_{H}(v)+1},\ldots,V_{d_{H}(v)+d_{G}(v)-1}$. We next give $\epsilon-d_{G}(v)+1$ new colors to the rest of open packings in $H$. The resulting function is an injective coloring of $F_{k}$ using
\begin{equation}\label{EQ1}
\chi_{i}(G)+\epsilon-d_{G}(v)+1=\chi_{i}(G)+\chi_{i}(H)-d_{H}(v)-d_{G}(v)+1=\chi_{i}(H)+1
\end{equation}
colors (since $\chi_{i}(G)-d_{G}(v)-d_{H}(v)=\vartheta=0$). We infer, in this case, that $\chi_{i}(F_{k})\in \{\chi_{i}(H),\chi_{i}(H)+1\}$ due to (\ref{max}).

\smallskip
\noindent
\textbf{Subcase 1.2.} $\vartheta>0$. \\
Let $g_{2}$ be an extension of $g_{1}$ such that the color $i$ is assigned to $V_{i-d_{G}(v)}$ for each $i=d_{G}(v)+d_{H}(v)+1,\ldots,\chi_{i}(G)$. In such a situation, $\varsigma=\chi_{i}(H)-\chi_{i}(G)+d_{G}(v)\geq1$ open packings in $H$ have not been colored under $g_{2}$. We consider two cases depending on $\chi_{i}(G)$ and $\chi_{i}(H)$.

\smallskip
\noindent
\textbf{Subcase 1.2.1.} $\chi_{i}(G)\geq \chi_{i}(H)$. \\
This shows that $\varsigma\leq d_{G}(v)$. Let $\varsigma\leq d_{G}(v)-1$. In such a situation, $g_{2}$ can be extended to $g_{3}$ by assigning $\varsigma$ colors from $\{2,\ldots,d_{G}(v)\}$ to the rest of open packings in $H$. The resulting coloring is injective and uses $\chi_{i}(G)$ colors. So, $\chi_{i}(F_{k})=\chi_{i}(G)$. If $\varsigma=d_{G}(v)$, then $g_{3}$ can be extended to $F_{k}$ by assigning a new color to the last open packing in $H$. Therefore, $\chi_{i}(F_{k})\in \{\chi_{i}(G),\chi_{i}(G)+1\}$.

\smallskip
\noindent
\textbf{Subcase 1.2.2.} $\chi_{i}(H)>\chi_{i}(G)$. \\
We then have $\varsigma>d_{G}(v)$. In this situation, let $g_{3}$ be an extension of $g_{2}$ that respectively assigsn the colors $2,\ldots,d_{G}(v)$ to $V_{\chi_{i}(G)-d_{G}(v)+1},\ldots,V_{\chi_{i}(G)-1}$. We now assign $\chi_{i}(H)-\chi_{i}(G)+1$ new colors to the remaining open packings in $H$. This leads to the existence of an injective coloring of $F_{k}$ with $\chi_{i}(H)+1$ colors. Therefore, $\chi_{i}(F_{k})\in \{\chi_{i}(H),\chi_{i}(H)+1\}$.

\medskip
\noindent
\textbf{Case 2.} $\chi_{i}(G)<d_{H}(v)+d_{G}(v)$. \\
Let first $\sigma=\chi_{i}(G)-d_{G}(v)=0$. Assume now that $\chi_{i}(H)=d_{H}(v)$. Let $g_1$ be an extension of $g$ which respectively assigns the colors $d_{G}(v)+1,d_{G}(v)+2,\ldots,d_{H}(v)$ to $V_{1}\setminus \{v\},V_{2},\ldots,V_{d_{H}(v)}$. Hence, $g_{1}$ is an injective coloring of $F_{k}$ with $d_{G}(v)+d_{H}(v)$ colors, and so $\chi_{i}(F_{k})=d_{G}(v)+d_{H}(v)$. Suppose now that $\chi_{i}(H)>d_{H}(v)$. If $\chi_{i}(H) < d_{G}(v)+d_{H}(v)$, then $g_{1}$ can be extended by assigning $\chi_{i}(H)-d_{H}(v)$ colors from $\{2,\ldots,d_{G}(v)\}$ to the rest of open packings in $H$. This gives an injective coloring of $F_{k}$ using $d_{G}(v)+d_{H}(v)$ colors, and therefore $\chi_{i}(F_{k})=d_{G}(v)+d_{H}(v)$. Now let $\chi_{i}(H)\geq d_{G}(v)+d_{H}(v)$. In such a situation, as an extension of $g_{1}$, we first respectively give $2,\ldots,d_{G}(v)$ colors to $V_{d_{H}(v)+1},\ldots,V_{d_{G}(v)+d_{H}(v)-1}$. We next assign $\chi_{i}(H)-d_{G}(v)-d_{H}(v)+1$ new colors to the rest of open packings in $H$. This leads to an injective coloring of $F_{k}$ with $\chi_{i}(H)+1$ colors, and therefore $\chi_{i}(F_{k})\in \{\chi_{i}(H),\chi_{i}(H)+1\}$.

Consider now that $\sigma>0$. Note that $g$ can be extended to a function $g_{1}$ by respectively assigning $d_{G}(v)+1,d_{G}(v)+2,\ldots,\chi_{i}(G)$ to $V_{1}\setminus \{v\},V_{2},\ldots,V_{\sigma}$, as well as, $d_{H}(v)-\sigma$ new colors to $V_{\sigma+1},\ldots,V_{d_{H}(v)}$. If $\chi_{i}(H)=d_{H}(v)$, then this defines an injective coloring of $F_{k}$ using $\chi_{i}(G)+d_{H}(v)-\sigma=d_{G}(v)+d_{H}(v)$ colors. Therefore, $\chi_{i}(F_{k})=d_{G}(v)+d_{H}(v)$. So, let $\chi_{i}(H)>d_{H}(v)$. Again, we need to consider two more possibilities.

\smallskip
\noindent
\textbf{Subcase 2.1.} $\chi_{i}(H)<d_{G}(v)+d_{H}(v)$. \\
In view of this, let $g_{2}$ be an extension of $g_{1}$ to $F_{k}$ by giving $\chi_{i}(H)-d_{H}(v)$ colors from $\{2,\ldots,d_{G}(v)\}$ to the rest of open packings in $H$. This process injectively colors $F_{k}$ by $d_{G}(v)+d_{H}(v)$ colors. So, we again have $\chi_{i}(F_{k})=d_{G}(v)+d_{H}(v)$.

\smallskip
\noindent
\textbf{Subcase 2.2.} $\chi_{i}(H)\geq d_{G}(v)+d_{H}(v)$. \\
Respectively assigning the colors $2,\ldots,d_{G}(v)$ to $V_{d_{H}(v)+1},\ldots,V_{d_{H}(v)+d_{G}(v)-1}$, as well as, $\chi_{i}(H)-d_{G}(v)-d_{H}(v)+1$ new colors to the rest of open packings in $H$, we obtain an extension of $g_{1}$ to $F_{k}$. It is easy to see that the resulting function is an injective coloring using $\chi_{i}(H)+1$ colors. Therefore, $\chi_{i}(F_{k})\in \{\chi_{i}(H),\chi_{i}(H)+1\}$.
\end{proof}

\begin{lemma}\label{L4}
Let $g(v)\in[d_{G}(v)]$ and $h(v)\notin h\big{(}N_{H}(v)\big{)}$ for each $\chi_{i}(H)$-function $h$. Then,
$$\chi_{i}(F_{k})\in \big{\{}\chi_{i}(G),\chi_{i}(H),\chi_{i}(G)+1,\chi_{i}(H)+1,d_{G}(v)+d_{H}(v),d_{G}(v)+d_{H}(v)+1\big{\}}.$$
\end{lemma}

\begin{proof}
We first observe, by the assumption given in the statement of the lemma, that $\chi_{i}(H)\geq d_{H}(v)+1$. For the sake of simplicity, we let $N_{H}(v)\subseteq V_{1}\cup \ldots \cup V_{d_{H}(v)}$, $v\in V_{d_{H}(v)+1}$ and $v\in U_{1}$. We again need to distinguish two possibilities depending on $\chi_{i}(G)$, $d_{G}(v)$ and $d_{H}(v)$.

\medskip
\noindent
\textbf{Case 1.} $\chi_{i}(G)\geq d_{H}(v)+d_{G}(v)+1$. \\
Assume first that $\chi_{i}(H)=d_{H}(v)+1$. The extension $g_{1}$, of $g$, that respectively assigns $d_{G}(v)+1,\ldots,d_{G}(v)+d_{H}(v)+1$ to $V_{1},\ldots,V_{d_{H}(v)+1}\setminus \{v\}$ is an injective coloring of $F_{k}$ with $\chi_{i}(G)$ colors. Thus, $\chi_{i}(F_{k})=\chi_{i}(G)$ (note that if $V_{d_{H}(v)+1}\setminus \{v\}=\emptyset$, then the color $d_{G}(v)+d_{H}(v)+1$ is not used in $H$). So, let $\chi_{i}(H)>d_{H}(v)+1$. Again, we need to consider two more cases.

\smallskip
\noindent
\textbf{Subcase 1.2.} $\lambda=\chi_{i}(G)-d_{G}(v)-d_{H}(v)-1\geq \chi_{i}(H)-d_{H}(v)-1=\varepsilon$. \\
We observe that a function $g_{2}$ defined, as an extension of $g_{1}$, by assigning the color $i$ to $V_{i-d_{G}(v)}$ for every $i=d_{G}(v)+d_{H}(v)+2,\ldots,\chi_{i}(H)+d_{G}(v)$ is an injective coloring of $F_{k}$ with $\chi_{i}(G)$ colors. Therefore, $\chi_{i}(F_{k})=\chi_{i}(G)$.

\smallskip
\noindent
\textbf{Subcase 2.2.} $\lambda<\varepsilon$. \\
Suppose first that $\lambda=0$. There exist two possibilities depending on $\chi_{i}(G)$ and $\chi_{i}(H)$.

\smallskip
\noindent
\textbf{Subcase 2.2.1.} $\chi_{i}(G)\geq \chi_{i}(H)$. \\
This implies that $\varepsilon\leq d_{G}(v)$ by taking $\lambda=0$ into account. If $\varepsilon\leq d_{G}(v)-1$, then an extension of $g_{1}$ that assigns $\varepsilon$ colors from $\{2,\ldots,d_{G}(v)\}$ to $V_{d_{H}(v)+2},\ldots,V_{\chi_{i}(H)}$ gives an injective coloring of $F_{k}$ with $\chi_{i}(G)$ colors. Therefore, $\chi_{i}(F_{k})=\chi_{i}(G)$. Now let $\varepsilon=d_{G}(v)$. In such a situation, we first respectively assign $2,\ldots,d_{G}(v)$ colors to $V_{d_{H}(v)+2},\ldots,V_{\chi_{i}(H)-1}$, and a new color to $V_{\chi_{i}(H)}$. The resulting function is an injective coloring of $F_{k}$ with $\chi_{i}(G)+1$ colors. Therefore, $\chi_{i}(F_{k})\in \{\chi_{i}(G),\chi_{i}(G)+1\}$.

\smallskip
\noindent
\textbf{Subcase 2.2.1.} $\chi_{i}(H)>\chi_{i}(G)$. \\
We then have $\varepsilon>d_{G}(v)$ since $\lambda=0$. Let $g_{2}$ be an extension of $g_{1}$ that respectively assigns $2,\ldots,d_{G}(v)$ to $V_{d_{H}(v)+2},\ldots,V_{d_{G}(v)+d_{H}(v)}$. Notice that $\chi_{i}(H)-d_{G}(v)-d_{H}(v)$ open packings in $H$ have not received colors under $g_{2}$. In such a situation, we obtain an injective coloring of $F_{k}$ with $\chi_{i}(H)+1$ colors by assigning $\chi_{i}(H)-d_{G}(v)-d_{H}(v)$ new colors to the remaining open packings. Hence, $\chi_{i}(F_{k})\in \{\chi_{i}(H),\chi_{i}(H)+1\}$.

\smallskip
Assume now that $\lambda>0$. Let $g_{2}$ be an extension of $g_{1}$ that assigns the color $i$ to $V_{i-d_{G}(v)}$ when $i\in \{d_{G}(v)+d_{H}(v)+2,\ldots,\chi_{i}(G)\}$. We note that $\upsilon=\chi_{i}(H)-\chi_{i}(G)+d_{G}(v)$ open packings in $H$ have not received colors under $g_{2}$. If $\chi_{i}(H)\geq \chi_{i}(G)$, then we respectively assign the colors $2,\ldots,d_{G}(v)$ to $V_{\chi_{i}(G)-d_{G}(v)+1},\ldots,V_{\chi_{i}(G)-1}$. Also, we give $\chi_{i}(H)-\chi_{i}(G)+1$ new colors to the rest of open packings in $H$. This leads to an injective coloring of $F_{k}$ using $\chi_{i}(H)+1$ colors, and hence $\chi_{i}(F_{k})\in \{\chi_{i}(H),\chi_{i}(H)+1\}$.

On the other hand, if $\chi_{i}(G)>\chi_{i}(H)$, then $\upsilon<d_{G}(v)$. In such a case, $g_{2}$ can be extended to $F_{k}$ by assigning $\upsilon$ colors from $\{2,\ldots,d_{G}(v)\}$ to the rest of open packings in $H$. This defines an injective coloring of $F_{k}$ with $\chi_{i}(G)$ colors, and therefore $\chi_{i}(F_{k})=\chi_{i}(G)$.

\medskip
\noindent
\textbf{Case 2.} $\chi_{i}(G)<d_{H}(v)+d_{G}(v)+1$. \\
We need to distinguish two more possibilities depending on $\chi_{i}(G)-d_{G}(v)$.

\smallskip
\noindent
\textbf{Subcase 2.1.} $\chi_{i}(G)=d_{G}(v)$. \\
If $\chi_{i}(H)=d_{H}(v)+1$, then assume $g_{1}$ respectively assigns the colors $d_{G}(v)+1,\ldots,d_{G}(v)+d_{H}(v)+1$ to $V_{1},\ldots,V_{d_{H}(v)+1}\setminus \{v\}$ (if $V_{d_{H}(v)+1}\setminus \{v\}=\emptyset$, then the color $d_{G}(v)+d_{H}(v)+1$ is not used in $H$). This gives an injective coloring of $F_{k}$ using at most $d_{G}(v)+d_{H}(v)+1$ colors. This shows that $\chi_{i}(F_{k})\in \{d_{G}(v)+d_{H}(v),d_{G}(v)+d_{H}(v)+1\}$. Let $\chi_{i}(H)>d_{H}(v)+1$. If $\varepsilon<d_{G}(v)$, then $g_{1}$ can be extended as an injective coloring of $F_{k}$ by assigning $\varepsilon$ colors from $\{2,\ldots,d_{G}(v)\}$ to $V_{d_{H}(v)+2},\ldots,V_{\chi_{i}(H)}$. Therefore, $\chi_{i}(F_{k})\in \{d_{G}(v)+d_{H}(v),d_{G}(v)+d_{H}(v)+1\}$. So, we let $\varepsilon \geq d_{G}(v)$. Let $g_{2}$ be an extension of $g_{1}$ that respectively assigns the colors $2,\ldots,d_{G}(v)$ to $V_{d_{H}(v)+2},\ldots,V_{d_{G}(v)+d_{H}(v)}$. We now give $\chi_{i}(H)-d_{G}(v)-d_{H}(v)$ new colors to the rest of open packings in $H$. This process ends with an injective coloring of $F_{k}$ using $\chi_{i}(H)+1$ colors, and hence $\chi_{i}(F_{k})\in \{\chi_{i}(H),\chi_{i}(H)+1\}$.

\smallskip
\noindent
\textbf{Subcase 2.1.} $\chi_{i}(G)>d_{G}(v)$.\\
Let $\chi_{i}(H)=d_{H}(v)+1$. Assume $g_{1}$ is an extension of $g$ that respectively assigns $d_{G}(v)+1,\ldots,\chi_{i}(G)$ to $V_{1},\ldots,V_{\chi_{i}(G)-d_{G}(v)}$. We then give at most $\chi_{i}(H)-\chi_{i}(G)+d_{G}(v)$ new colors to $V_{\chi_{i}(G)-d_{G}(v)+1},\ldots,V_{d_{H}(v)+1}\setminus \{v\}$ (trivially, $V_{d_{H}(v)+1}\setminus \{v\}$ does not receive any color if $V_{d_{H}(v)+1}=\{v\}$). This results in an injective coloring of $F_{k}$ with at most $d_{G}(v)+d_{H}(v)+1$ colors. Therefore, $\chi_{i}(F_{k})\in \{d_{G}(v)+d_{H}(v),d_{G}(v)+d_{H}(v)+1\}$.

Let $\chi_{i}(H)>d_{H}(v)+1$. Consider now that $g_{2}$ extends $g_{1}$ by giving $d_{H}(v)+1-\chi_{i}(G)+d_{G}(v)$ new colors to $V_{\chi_{i}(G)-d_{G}(v)+1},\ldots,V_{d_{H}(v)+1}\setminus \{v\}$ (note that $V_{d_{H}(v)+1}\setminus \{v\}$ does not receive any color if $V_{d_{H}(v)+1}=\{v\}$). In such a situation, $\varepsilon$ open packings in $H$ have not received colors under $g_{2}$. If $\varepsilon<d_{G}(v)$, then we assign $\varepsilon$ colors from $\{2,\ldots,d_{G}(v)\}$ to the rest of open packings in $H$. This defines an injective coloring of $F_{k}$ with at most $d_{G}(v)+d_{H}(v)+1$ colors. Hence, $\chi_{i}(F_{k})\in \{d_{G}(v)+d_{H}(v),d_{G}(v)+d_{H}(v)+1\}$. So, we let $\varepsilon\geq d_{G}(v)$. In this situation, we extend $g_{2}$ by respectively assigning the colors $2,\ldots,d_{G}(v)$ to $V_{d_{H}(v)+2},\ldots,V_{d_{G}(v)+d_{H}(v)}$, as well as, $\chi_{i}(H)-d_{G}(v)-d_{H}(v)$ new colors to the rest of open packings in $H$. This leads to the existence of an injective coloring of $F_{k}$ with at most $\chi_{i}(H)+1$ colors. Thus, $\chi_{i}(F_{k})\in \{\chi_{i}(H),\chi_{i}(H)+1\}$. This completes the proof.
\end{proof}

Altogether, Lemmas \ref{L1}--\ref{L4} imply that
\begin{equation}\label{four-lemma}
\chi_{i}(F_{k})\in \big{\{}\chi_{i}(G),\chi_{i}(H),\chi_{i}(G)+1,\chi_{i}(H)+1,d_{G}(v)+d_{H}(v),d_{G}(v)+d_{H}(v)+1\big{\}}
\end{equation}
for each $k\in[n]$ and any graphs $G$ and $H$ with $v\in V(H)$.

For every $k\in[n]$, by renaming the colors assigned to $V(H_{k})\setminus \{v_{k}\}$ if necessary, we may assume that the optimal injective coloring of $F_{k}$ uses the colors from $[\chi_{i}(F_{k})]$. Recall that such an injective coloring uses the colors from $[\chi_{i}(G)]$ in $G$.

We observe that $V(G\circ_{v}H)=\bigcup_{k=1}^{n}V(F_{k})$ and that $V(F_{i})\cap V(F_{j})=V(G)$ for every distinct $i,j\in[n]$. Assume in the rest that $F\in \{F_{1},\ldots,F_{n}\}$ has the property that $$\chi_{i}(F)=\max_{k\in[n]} \{\chi_{i}(F_{k})\}.$$
With these notations in mind, we prove the following simple but useful lemma.

\begin{lemma}\label{L5}
$\chi_{i}(F)\in \big{\{}\chi_{i}(G),\chi_{i}(H),\chi_{i}(G)+1,\chi_{i}(G)+1,\Delta(G)+d_{H}(v),\Delta(G)+d_{H}(v)+1\big{\}}$.
\end{lemma}

\begin{proof}
Since $F=F_{k}$ for some $k\in[n]$, we have
\begin{equation*}
\chi_{i}(F)\in \big{\{}\chi_{i}(G),\chi_{i}(H),\chi_{i}(G)+1,\chi_{i}(H)+1,d_{G}(v_{k})+d_{H}(v),d_{G}(v_{k})+d_{H}(v)+1\big{\}}.
\end{equation*}
by (\ref{four-lemma}). Moreover, for each vertex $v_{j}$ in $G$ of maximum degree, (\ref{four-lemma}) implies that
\begin{equation*}
\chi_{i}(F_{j})\in \big{\{}\chi_{i}(G),\chi_{i}(H),\chi_{i}(G)+1,\chi_{i}(H)+1,\Delta(G)+d_{H}(v),\Delta(G)+d_{H}(v)+1\big{\}}=M_{\Delta}.
\end{equation*}

Suppose to the contrary that $\chi_{i}(F)\notin M_{\Delta}$. This necessarily implies that $\chi_{i}(F)\geq \chi_{i}(G)+2$ and that $\chi_{i}(F)\geq \chi_{i}(H)+2$. If $\chi_{i}(F)=d_{G}(v_{k})+d_{H}(v)$, then $d_{G}(v_{k})+d_{H}(v)=\chi_{i}(F)\geq \chi_{i}(F_{j})\geq \Delta(G)+d_{H}(v)$. This necessarily implies that $d_{G}(v_{k})=\Delta(G)$, contradicting the supposition $\chi_{i}(F)\notin M_{\Delta}$. Therefore, $\chi_{i}(F)=d_{G}(v_{k})+d_{H}(v)+1$. Similarly, we have $d_{G}(v_{k})+d_{H}(v)+1=\chi_{i}(F)\geq \chi_{i}(F_{j})\geq \Delta(G)+d_{H}(v)$. Hence, $d_{G}(v_{k})+1\in \{\Delta(G),\Delta(G)+1\}$, contradicting $\chi_{i}(F)\notin M_{\Delta}$. So, the statement of the lemma holds.
\end{proof}

We are now in a position to prove the main result of this section.

\begin{theorem}\label{rooted}
For any graph $G$ and any graph $H$ with root $v\in V(H)$,
$$\chi_{i}(G\circ_{v}H)\in \big{\{}\chi_{i}(G),\chi_{i}(H),\chi_{i}(G)+1,\chi_{i}(H)+1,\Delta(G)+d_{H}(v),\Delta(G)+d_{H}(v)+1\big{\}}.$$
\end{theorem}

\begin{proof}
Note that $\chi_{i}(G\circ_{v}H)\geq \chi_{i}(F)$ as $F=F_{k}$ is a subgraph of $G\circ_{v}H$. Let $f_{j}$ be a $\chi_{i}(F_{j})$-coloring for each $j\in[n]$, as constructed along the proofs of Lemmas \ref{L1}--\ref{L4}. We now define $f$ on $V(G\circ_{v}H)$, as an extension of $g$, by $f(x)=f_{j}(x)$ when $x\in V(F_{j})$. Notice that $f$ is well-defined because $f_{i}(x)=f_{j}(x)$ for each $x\in V(G)$ and every distinct $i,j\in[n]$.

Suppose to the contrary that there exist distinct vertices $x,y,z\in V(G\circ_{v}H)$ such that $y,z\in N_{G\circ_{v}H}(x)$ and $f(y)=f(z)$. Since the restrictions of $f$ to $V(G)$ and $V(H_{j})$, are injective colorings for each $j\in[n]$, it follows that neither ``$y,z\in V(G)$" nor ``$y,z\in V(H_{j})$ for some $j\in[n]$" happens. So, without loss of generality, we may assume that $z\in V(G)$ and $y\in V(H_{j})$ for some $j\in[n]$. By the structure, this necessarily implies that $x=v=v_{j}$. This contradicts the fact that $f_{j}$ is an injective coloring of $F_{j}$. So, we deduce that $f$ is an injective coloring of $G\circ_{v}H$.

Recall that for every $j\in[n]$, $f_{j}$ assigns the colors in $[\chi_{i}(F_{j})]$ so as to injectively color $F_{j}$. Due to this fact, we observe that $f$ assigns $\chi_{i}(F)$ colors to $V(G\circ_{v}H)$, and hence $\chi_{i}(G\circ_{v}H)\leq \chi_{i}(F)$. This leads to the desired inequality $\chi_{i}(G\circ_{v}H)=\chi_{i}(F)\in \{\chi_{i}(G),\chi_{i}(H),\chi_{i}(G)+1,\chi_{i}(H)+1,\Delta(G)+d_{H}(v),\Delta(G)+d_{H}(v)+1\}$ by Lemma \ref{L5}.
\end{proof}

It can be readily seen that the six possible values for $\chi_{i}(G\circ_{v}H)$ given in Theorem \ref{rooted} can indeed be presented in the following way.

\begin{corollary}
For any graph $G$ and any graph $H$ with root $v\in V(H)$,
{\small
$$\max\left\{\chi_{i}(G),\chi_{i}(H),\Delta(G)+d_{H}(v)\right\}\le \chi_{i}(G\circ_{v}H)\le \max\left\{\chi_{i}(G),\chi_{i}(H),\Delta(G)+d_{H}(v)\right\}+1.$$}
\end{corollary}

In what follows, we show that $\chi_{i}(G\circ_{v}H)$ can indeed reach each of the six values appearing in the closed formula of Theorem \ref{rooted}, depending on our choice for $G$ and $H$. Suppose first that $G\cong K_{n}$ on $n\geq3$ vertices and let $H$ be obtained from $K_{m}$, with $m\geq3$, by joining a new vertex $v$ to only one vertex of $K_{m}$. It is then readily checked that $\chi_{i}(G\circ_{v}H)=n=\chi_{i}(G)$ if $n\geq m$, and that $\chi_{i}(G\circ_{v}H)=m=\chi_{i}(H)$ if $m>n$. Let $G\cong K_{1,a}$ and $H\cong K_{1,b}$ for some integers $a,b\geq1$. It is then clear that $\chi_{i}(G\circ_{v}H)=a+b=\Delta(G)+d_{H}(v)$, in which $v$ is the center of $H$. This in particular shows that $\chi_{i}(G\circ_{v}H)=a+b=\chi_{i}(G)+1$ when $b=1$. Let $G\cong C_{4n}$ and $H\cong K_{m}$ for some integers $n\geq1$ and $m\geq3$. Recall that $\chi_{i}(C_{4n})=2$. We then have $\chi_{i}(G\circ_{v}H)=m+1=\chi_{i}(H)+1$. Finally, let $G\cong K_{n}$ and $H\cong K_{m}$ for some integers $m,n\geq3$, in which $v$ is any vertex of $K_{m}$. It is easy to see that $\chi_{i}(G\circ_{v}H)=n+m-1=\Delta(G)+d_{H}(v)+1$.

\subsection{Corona products viewed as rooted products}

Let $G$ and $H$ be graphs where $V(G)=\{v_1,\ldots,v_{n}\}$. The \textit{corona product} $G\odot H$ of the graphs $G$ and $H$ is obtained from the disjoint union of $G$ and $n$ disjoint copies of $H$, say $H_1,\ldots,H_{n}$, such that $v_i\in V(G)$ is adjacent to all vertices of $H_i$ for each $i\in[n]$. Recall that the \textit{join} of graphs $G$ and $H$, written $G\vee H$, is a graph obtained from the disjoint union $G$ and $H$ by adding the edges $\{gh\mid g\in V(G)\ \mbox{and}\ h\in V(H)\}$.

As an immediate consequence of Theorem \ref{rooted}, we obtain the closed formula for the injective chromatic number of corona product graphs given in \cite{SSY}. To do so, we need some routine observations. Let $G$ and $H$ have no isolated vertices. Moreover, we may assume that they are connected. We observe that $G\odot H$ is isomorphic to $G\circ_{v}(K_{1}\vee H)$, in which the root $v$ is the unique vertex of $K_{1}$. Due to this and the fact that $\chi_{i}(K_{1}\vee H)=|V(H)|+1=d_{K_{1}\vee H}(v)+1$, Theorem \ref{rooted} implies that $\chi_{i}(G\odot H)$ belongs to the set
\begin{equation}\label{R1}
\big{\{}\chi_{i}(G),|V(H)|+1,\chi_{i}(G)+1,|V(H)|+2,\Delta(G)+|V(H)|,\Delta(G)+|V(H)|+1\big{\}}.
\end{equation}

On the other hand, it is a routine matter to see that $\chi_{i}(K_{2}\odot H)=|V(H)|+1=\Delta(K_{2})+|V(H)|$. In view of this, we may assume that $\Delta(G)\geq2$. Since $\chi_{i}(G\odot H)\geq \Delta(G\odot H)=\Delta(G)+|V(H)|$, it follows that $|V(H)|+1$ can be excluded from the set in (\ref{R1}). By a similar fashion, $|V(H)|+2$ can also be excluded when $\Delta(G)\geq2$.

We observe, in view of Lemmas \ref{L1}--\ref{L4}, that the equality $\chi_{i}(G\odot H)=\chi_{i}(G)+1$ may only occur in Lemma \ref{L4} (note that $h(v)\notin h\big{(}N_{K_{1}\vee H}(v)\big{)}$, for each $\chi_{i}(K_{1}\vee H)$-function $h$, since $H$ has no isolated vertices). Suppose now that $\chi_{i}(G\odot H)=\chi_{i}(G)+1$. By the proof of Lemma \ref{L4}, it only happens when $\chi_{i}(K_{1}\vee H)>|V(H)|+1$, which is a contradiction.

\begin{corollary}\label{SSYr}\emph{(\cite{SSY})}
For any graphs $G$ and $H$ with no isolated vertices,
\begin{center}
$\chi_{i}(G\odot H)\in \big{\{}\chi_{i}(G),|V(H)|+\Delta(G),|V(H)|+\Delta(G)+1\big{\}}$.
\end{center}
\end{corollary}

\section{Kneser graphs}
\label{sec:Kneser}

For positive integers $n$ and $r$, where $n\geq 2r$, the {\em Kneser graph} $K(n,r)$ has the $r$-subsets of an $n$-set as its vertices and two vertices are adjacent in $K(n,r)$ if the corresponding sets are disjoint. Kneser graphs are among the most studies classes of graphs, since the two classical results concerning their independence and chromatic numbers were proved roughly half a century ago~\cite{ekr-61,lov-78}. In two recent papers~\cite{bcdh,coto}, the $2$-packing numbers of Kneser graphs were studied, and we will use some results from these papers for finding the open packing numbers and discussing their perfect injectively colorability.

It is well known and easy to see that ${\rm diam}(K(n,r))=2$ if and only if $n\ge 3r-1$. This immediately gives $\rho_2(K(n,r))=1$ if and only if $n\ge 3r-1$. Now, we invoke the result about perfect injectively colorable graphs with diameter $2$ from~\cite{BSY}.

\begin{proposition} {\rm (\cite[Proposition 13]{BSY})}
If $G$ is a graph with $\diam(G) = 2$, then $G$ is a perfect injectively colorable graph if and only if either each edge of $G$ lies in a triangle, or there exists a perfect matching $M$ in $G$ such that no edge of $M$ lies in a triangle.
\label{prp:diam2}
\end{proposition}

If $n\ge 3r$, then every edge of $K(n,r)$ clearly lies in a triangle, hence by Proposition~\ref{prp:diam2}, it is a perfect injectively colorable graph. Now, if $n=3r-1$, we claim that $K(n,r)$ has a perfect matching. Indeed, one can see this by using the recent result from~\cite{mmn-2023+} that all Kneser graphs with the sole exception of the Petersen graph are Hamiltonian, and the fact that $\binom{3r-1}{r}$ is an even number. Thus, since $K(n,r)$ has no triangles, we infer by Proposition~\ref{prp:diam2} again, that $K(n,r)$ is a perfect injectively colorable graph.  We state the obtained remarks as follows.

\begin{observation}\label{obser}
If $n\ge 3r-1$, then $K(n,r)$ is a perfect injectively colorable graph.
\end{observation}

In~\cite{coto}, the authors studied the Kneser graphs $K(n,r)$ with $n=3r-2$, which are in a sense the closest to diameter-$2$ Kneser graphs. They obtained the exact values of the $2$-packing number for all these Kneser graphs as follows:
\begin{equation}\label{2-pack}
\rho_2\big{(}K(3r-2,r)\big{)}=
        \begin{cases}
            7 & \mbox{if}\ r=3,\\
            5 & \mbox{if}\ r=4,\\
            3 & \mbox{if}\ r\ge 5.
        \end{cases}
\end{equation}

Let $S$ be an open packing of $K(3r-2,r)$, and suppose that $S$ is not a $2$-packing. Therefore, there exist vertices $u$ and $v$ in $S$, which are adjacent in $K(n,r)$. Without loss of generality, let $u=[r]$ and $v=\{r+1,\ldots,2r\}$. Since $\rho(G)\leq \rho^{o}(G)$ for all graphs $G$, the equality (\ref{2-pack}) implies that there exists a vertex $w\in S\setminus \{u,v\}$. Clearly, $d(w,u)>2$ and d$(w,v)>2$. In particular, $w\cap u\neq \emptyset$. Suppose that $|w\cap u|\geq 2$. Then $|w\cup u|\leq2r-2$. Since $n=3r-2$, we infer that there exists a vertex $x$ which is adjacent to both $w$ and $u$, a contradiction to d$(w,u)>2$. Therefore, $|w\cap u|=1$, and by a similar argument $|w\cap v|=1$. This yields $\{2r+1,\ldots,3r-2\}\subset w$. Hence, if $w'$ is any other vertex in $S$, we infer that $|w\cap w'|\geq r-2$. In the case $r>3$, this gives $|w\cap w'|\ge 2$ implying $|w\cup w'|\leq r+2$, yet this yields that $w$ and $w'$ have a common neighbor, which is impossible. We have thus shown that every maximum open packing is a $2$-packing in $K(3r-2,r)$ for $r>3$, and hence $\rho^o\big{(}K(3r-2,r)\big{)}=\rho_2\big{(}K(3r-2,r)\big{)}$ in this case.

If $r=3$, then $\rho^o\big{(}K(3r-2,r)\big{)}\geq7$ by (\ref{2-pack}). Notice that every vertex in $S\setminus \{u,v\}$ is of the form $w=\{a,b,7\}$, in which $a\in u$ and $b\in v$. Moreover, $|w\cap w'|=1$ for any two vertices $w,w'\in S\setminus \{u,v\}$ as d$(w,w')>2$. Hence, we can add at most three vertices to $u=\{1,2,3\}$ and $v=\{4,5,6\}$ in order to get an open packing. This contradicts that fact that $\rho^o\big{(}K(3r-2,r)\big{)}\geq7$. In fact, we have proved that every maximum open packing in $K(3r-2,r)$ is a $2$-packing. In particular, we have
\begin{equation}\label{open pack}
\rho^o\big{(}K(3r-2,r)\big{)}=
\begin{cases}
7 &  \mbox{if}\ r=3,\\
5 &  \mbox{if}\ r=4,\\
3 &  \mbox{if}\ r\ge 5.
\end{cases}
\end{equation}
by the equality (\ref{2-pack}).

Note that Observation \ref{obser} is in a sense best possible as there exists a Kneser graph $K(n,r)$ with $n=3r-2$, for some positive integer $r$, which is not perfect injectively colorable.

\begin{proposition}
Kneser graph $K(7,3)$ is not perfect injectively colorable.
\end{proposition}
\begin{proof}
Note that a maximum $2$-packing $P$ in $K(7,3)$ consists of seven $3$-subsets of the set $[7]$, where each $i\in [7]$ appears in exactly three of these seven subsets. Thus, $P$ corresponds to the Fano plane.

Suppose to the contrary that $K(7,3)$ admits an injective coloring such that each color class is a maximum open packing. Hence, all color classes are of cardinality $7$ by (\ref{open pack}), and by the remark preceding the proposition, we infer that every color class is a $2$-packing of cardinality $7$. Therefore, there exists a $2$-distance coloring of $K(7,3)$ with $5$ colors such that each color class has $7$ vertices. In particular, we infer $\chi_2(K(7,3)=5$. This is a contradiction as the fact that one cannot partition $V\big{(}K(7,3)\big{)}$ into five Fano planes goes back to Cayley~\cite{ca-1850}. Therefore, $K(7,3)$ is not a perfect injectively colorable graph.
\end{proof}

We remark that the exact value $\chi_2(K(7,3))=6$ was proved in~\cite{kn-2004},

\section{Concluding remarks}
\label{sec:conclud}

Edge clique covers have been extensively investigated so far, see the survey~\cite{schwartz-2022}. On the other hand, the concept of sparse edge clique covers, which turned out to be very useful for our purpose, seems to be a new notion. We believe that such notion deserves an independent interest.

In Section~\ref{sec:Sierpinski} we have briefly considered the generalized Sierpi\'nski graphs over cycles, that is, the graphs $S_{C_k}^n$. The results presented indicate that an investigation of the injective chromatic number of generalized Sierpi\'nski graphs deserves attention, in particular describing those that are perfect injectively colorable. Notice that this task also requires the study of the open packing number of generalized Sierpi\'nski graphs.

In Section \ref{sec:rooted} we have considered the rooted product graphs that can be seen as an instance of the operation called Sierpi\'nski product (see \cite{kpzz-2022}). In this sense, it is of interest to continue investigating the injective chromatic number of other Sierpi\'nski products. In addition, the open packing number of such graphs is worthy of attention.

In Section~\ref{sec:Kneser}, we have shown that Kneser graphs $K(n,r)$ are perfect injectively colorable as soon as $n\ge 3r-1$, and that $K(7,3)$ is not a perfect injectively colorable graph. The latter graph is the only Kneser graph for which we know that it is not perfect injectively colorable, and it would be interesting to determine for which $r>3$ graphs $K(3r-2,r)$ are (not) perfect injectively colorable. The same question can be posed for the odd graphs (Kneser graphs of the form $K(2r+1,r)$) and Kneser graphs in general.

\section*{Acknowledgements}

This work was supported by the Slovenian Research and Innovation Agency ARIS (research core funding P1-0297 and projects J1-3002, J1-4008, N1-0285, N1-0218). I.\ G.\ Yero has been partially supported by the Spanish Ministry of Science and Innovation through the grant PID2022-139543OB-C41.

\section*{Declaration of interests}

The authors declare that they have no known competing financial interests or personal relationships that could have appeared to influence the work reported in this paper.

\section*{Data availability}
Our manuscript has no associated data.


\begin{thebibliography}{99}

\bibitem{balakrishnan-2022}
K.~Balakrishnan, M.~Changat, A.~M.~Hinz, and D.~S.~Lekha,
{\em The median of {S}ierpi\'nski graphs},
Discrete Appl.\ Math.\ {\bf 319} (2022), 159--170.
\bibitem{bcdh} B.~Bre\v{s}ar, M.~G.~Cornet, T.~Dravec, and M.~A.~Henning, {\em $k$-Domination invariants on Kneser graphs}, arXiv:2312.15464 [math.CO] (24 Dec 2023).
\bibitem {bkr2} B.~Bre\v{s}ar, S.~Klav\v{z}ar, and D.~F.~Rall, {\em Packings in bipartite prisms and hypercubes}, Discrete Math. {\bf 347} (2024), 113875.
\bibitem {BSY} B.~Bre\v{s}ar, B.~Samadi, and I.~G.~Yero, {\em Injective coloring of graphs revisited}, Discrete Math. {\bf 346} (2023), 113348.
\bibitem{bu-2009} Y.~Bu, D.~Chen, A.~Raspaud, and W.~Wang, {\em Injective coloring of planar graphs}, Discrete Appl. Math. {\bf 157} (2009), 663--672.

\bibitem{ca-1850} A.~Cayley, {\em On the triadic arrangements on seven and fifteen things}, London, Edinburgh and Dublin Philos.\ mag.\ and J.~Sci. {\bf 37} (1850), 50-53.

\bibitem{chu-2024}
H.~Chu and S.-R.~Kim,
{\em Competition graphs of degree bounded digraphs},
Discrete Appl.\ Math.\ {\bf 343} (2024), 106--114.

\bibitem{cky-2010} D.~Cranston, S.~J.~Kim, and G.~Yu, {\em Injective colorings of sparse graphs}, Discrete Math. {\bf 310} (2010), 2965--2973.
\bibitem{coto} M.~G.~Cornet and P.~Torres, {\em $k$-tuple domination on Kneser graphs}, arXiv:2308.15603 [math.CO] (5 Apr 2024).

\bibitem{deng-2021}
F.~Deng, Z.~Shao, and A.~Vesel,
{\em On the packing coloring of base-$3$ {S}ierpi\'nski graphs and   {$H$}-graphs},
Aequationes Math.\ {\bf 95} (2021), 329--341.

\bibitem{ekr-61} P.~Erd\H{o}s, C.~Ko, and R.~Rado, {\em Intersection theorem for system of finite sets}, Quart. J. Math. {\bf 12} (1961), 313--318.
\bibitem{GM} C.~D.~Godsil and B.~D.~McKay, {\em A new graph product and its spectrum}, Bull. Austral. Math. Soc. {\bf 18} (1978), 21--28.
\bibitem{gravier-2011}
S.~Gravier, M.~Kov\v se, and A.~Parreau,
 {\em eneralized Sierpi\'nski graphs}, in: Posters at EuroComb'11, Budapest, August 29-September 2, 2011,
\url{http://www.renyi.hu/conferences/ec11/posters/parreau.pdf} (2024--08--21).

\bibitem {hkss} G.~Hahn, J.~Kratochv\'{i}l, J.~\v{S}ir\'{a}\v{n}, and D.~Sotteau, {\em On the injective chromatic number of graphs}, Discrete Math. {\bf 256} (2002), 179--192.

\bibitem{hs-1999} M.A.~Henning, P.J.~Slater, Open packing in graphs,
J.\ Comb.\ Math.\ Comb.\ Comput.\ 28 (1999), 5--18.
\bibitem{hinz-2017}
A.~M.~Hinz, S.~Klav\v{z}ar, and S.~S.~Zemlji\v{c},
{\em A survey and classification of Sierpi\'nski-type graphs},
Discrete Appl.\ Math.\ {\bf 217} (2017), 565--600.
\bibitem{HP-2012} A.~M.~Hinz and D.~Parisse, {\em Coloring Hanoi and Sierpi\'nski graphs},  Discrete Math.\ {\bf 312} (2012), 1521--1535.

\bibitem{jakovac-2014}
M.~Jakovac,
{\em A $2$-parametric generalization of Sierpi\'nski gasket graphs},
Ars Combin.\ {\bf 116} (2014), 395--405.

\bibitem {JT} T.~R.~Jensen and B.~Toft, Graph coloring problems, John Wiley \& Sons, 2011.
\bibitem{kn-2004} S.~J.~Kim and K.~Nakprasit, {\em On the chromatic number of the square of the Kneser graph $K(2k+1,k)$}, Graphs Combin. {\bf 20} (2004), 79--90.
\bibitem{klavzar-1997}
S.~Klav\v zar and U.~Milutinovi\'c,
{\em Graphs $S(n,k)$ and a variant of the Tower of Hanoi problem},
Czechoslovak Math.\ J.\ {\bf 47} (1997), 95--104.
\bibitem{korze-2019}
D.~Kor\v{z}e and A.~Vesel,
{\em Packing coloring of generalized Sierpi\'nski graphs},
Discrete Math.\ Theor.\ Comput.\ Sci.\ {\bf 21} (2019), 7.

\bibitem{kpzz-2022} J.~Kovi\v{c}, T.~Pisanski, S.~S.~Zemlji\v{c}, and A.~\v{Z}itnik, {\em The Sierpi\'{n}ski product of graphs}, Ars Math.\ Contemp.\ {\bf 23} (2023), 1.
\bibitem{lov-78} L.~Lov\'{a}sz, Kneser's Conjecture, {\em Chromatic Numbers and Homotopy}, J. Combin.  Theory Ser.  A {\bf 25} (1978), 319--324.
\bibitem{lst-2009} B.~Lu\v{z}ar, R.~\v{S}krekovski, and M.~Tancer, {\em Injective colorings of planar graphs with few colors}, Discrete Math. {\bf 309} (2009), 5636--5649.

\bibitem{menon-2023}
M.~K.~Menon, M.~R.~Chithra, and K.~S.~Savitha,
{\em Security in {S}ierpi\'nski graphs},
Discrete Appl.\ Math.\ {\bf 328} (2023), 10--15.

\bibitem{mmn-2023+} A.~Merino, T.~M\"{u}tze, and Namrata,
\textit{Kneser graphs are Hamiltonian},
S{TOC}'23---{P}roceedings of the 55th {A}nnual {ACM} {S}ymposium on {T}heory of {C}omputing, ACM, New York (2023), 963--970.

\bibitem{nguyen-2024}
T.~Nguyen, A.~Scott, P.~Seymour, and S.~Thomass\'e,
{\em Clique covers of {$H$}-free graphs},
European J.\ Combin.\ {\bf 118} (2024), 103909.

\bibitem {pp} B.~S.~Panda and Priyamvada, {\em Injective coloring of some subclasses of bipartite graphs and chordal graphs}, Discrete Appl. Math. {\bf 291} (2021), 68--87.
\bibitem {SSY} B.~Samadi, N.~Soltankhah, and I.~G.~Yero, {\em Injective coloring of product graphs}, Bull. Malays. Math. Sci. Soc. {\bf 47} (2024), 86.

\bibitem{schwartz-2022}
S.~Schwartz,
{\em An overview of graph covering and partitioning},
Discrete Math.\ {\bf 345} (2022), 112884.

\bibitem{Sch} A.~J.~Schwenk, {\em Computing the characteristic polynomial of a graph}, Graphs Combin. (R. Bari and F. Harary, eds.), Lecture Notes in Math.\ 406, Springer, Berlin, (1974), 153--172.
\bibitem{sy} J.~Song and J.~Yue, {\em Injective coloring of some graph operations}, Appl. Math. Comput. {\bf 264} (2015), 279--283.
\bibitem {we} D.~B.~West, Introduction to Graph Theory (Second Edition), Prentice Hall, USA, 2001.

\end{thebibliography}
\end{document}